\documentclass[11pt]{article}

\usepackage[square, numbers]{natbib}

\usepackage{bm}
 
\usepackage{enumerate}
\usepackage{graphicx}
\usepackage[nice]{nicefrac}
\usepackage{algorithm, algcompatible}
\usepackage{mathrsfs}
\usepackage{mathtools}
\usepackage[labelfont=bf,font={small}]{caption}
\usepackage[labelfont=bf]{subcaption} 
\usepackage{wrapfig}
\usepackage{multirow}
 
\usepackage{fullpage}

\usepackage{amsmath,amsthm,amssymb,colonequals}

\usepackage{etoolbox}

\usepackage[bookmarks=true]{hyperref}
\usepackage[usenames,dvipsnames]{xcolor}
\definecolor{darkblue}{rgb}{0.0,0.0,0.55}
\definecolor{darkred}{rgb}{0.45,0.0,0.0}
\hypersetup{
  colorlinks = true,
  citecolor  = darkblue,
  linkcolor  = darkred,
  filecolor  = darkblue,
  urlcolor   = darkblue,
}

\usepackage{booktabs}
\usepackage{tikz}
\usetikzlibrary{shadows}
\definecolor{emph}{gray}{0.950317}
\usepackage[align=center,shadow=false,shadowsize=5pt,nobreak=true,framemethod=tikz,style=0,skipabove=5pt,skipbelow=1pt,innertopmargin=4pt,innerbottommargin=7pt,innerleftmargin=6pt,innerrightmargin=6pt,leftmargin=-2pt,rightmargin=-2pt]{mdframed}

\usepackage{graphicx}
\usepackage{epstopdf}
\usepackage[labelfont=bf,font=footnotesize]{caption}
\usepackage[labelfont=bf,font=footnotesize]{subcaption}
\usepackage{xcolor}
\usepackage{algorithmicx}
\usepackage{algpseudocode}
\usepackage{algorithm}

\usepackage{thmtools}

\newtheorem{theorem}{Theorem}[section]

\newtheorem{lemma}[theorem]{Lemma}
\newtheorem{corollary}[theorem]{Corollary}
\newtheorem{proposition}[theorem]{Proposition}
\newtheorem{claim}[theorem]{Claim}

\newtheorem{definition}{Definition}[section]
\newtheorem{assumption}[definition]{Assumption} 
\newtheorem{remark}[definition]{Remark}

\newcommand{\algmpc}{{\sc Forward-Switch}}

\newcommand{\BC}{{\sc Behavior Cloning}}

\newcommand{\forward}{{\sc Forward}}
\newtheorem*{alg}{\underline{\sc Forward Training Algorithm}}
\newtheorem*{alg2}{\underline{\algmpc}}

 \makeatletter
\def\namedlabel#1#2{\begingroup
   \def\@currentlabel{#2}%
   \label{#1}\endgroup
}
\makeatother

\newcommand{\cW}{\mathcal{W}}

\DeclareMathOperator*{\argmin}{arg\!\min}

\DeclareMathOperator*{\proj}{proj}

\newcommand{\X}{\ensuremath{\mathcal{X}}}
\newcommand{\U}{\ensuremath{\mathcal{U}}}

\newcommand{\W}{\ensuremath{\mathcal{W}}}

\newcommand{\R}{\ensuremath{\mathbb{R}}}

\newcommand{\B}{ {\sf {B}}}

\newcommand{\N}{\ensuremath{\mathbb{N}}}

\newcommand{\norm}[1]{\left\lVert #1 \right\rVert}

\newcommand{\E}{\mathbb{E}}

\renewcommand{\Pr}{\mathbb{P}}

\newcommand{\calD}{\mathcal{D}}

\newcommand{\st}{\text{s.t.}}

\numberwithin{equation}{section}

\renewcommand{\geq}{\geqslant}

\renewcommand{\leq}{\leqslant}

\renewcommand{\succeq}{\succcurlyeq}

\newcommand{\fclexp}{f_{\mathrm{cl}}^{\star}}

\newcommand{\minimize}{\mathrm{minimize}}

\newcommand{\D}{\mathcal{D}}

\newcommand{\nn}{\nonumber}
\newcommand{\xinit}{x}

\newcommand{\eps}{\varepsilon}
\newcommand{\expert}{\pi^\star} 
\newcommand{\traj}[2]{\varphi_t(#2;#1)}
\newcommand{\pibc}{\hat\pi_{\sf BC}}
\newcommand{\dis}{\hat w}

\newcommand{\trajs}[3]{\varphi^{\star}_{#1}(#3;#2)}

\newcommand{\trajj}[3]{\varphi_{#3}(#2;#1)}

\newcommand{\klqr}{K^{\sf \tiny lqr}}
\newcommand{\plqr}{P^{\sf \tiny lqr}}

\newcommand{\dlqr}{\pi^{\sf \tiny lqr}}
\newcommand{\mpc}{\pi^{\sf \tiny mpc}} 
\newcommand{\rmpc}{\boldsymbol{\pi}} 
\newcommand{\piref}{\bar{\pi}}  
\newcommand{\hmpc}{\hat{\pi}}
\newcommand{\tmpc}{\tilde{\pi}}

\newcommand{\xref}{\bar{x}}

\newcommand{\Xref}{\bar{\X}}
\newcommand{\Uref}{\bar{\U}}
\newcommand{\Xzeroref}{\bar{\X}_0}
\newcommand{\Xterref}{\bar{\X}_{\sf f}}
\newcommand{\Olqrref}{\bar{\mathcal{O}}_\infty}
\newcommand{\uref}{\bar{u}}
\newcommand{\ustar}{u^\star}

\newcommand{\hpi}{\hat\pi}

\newcommand{\pihat}{\hat{\pi}}

\newcommand{\vv}[1]{\bar{x}_{#1}}
\newcommand{\xx}[1]{\hat{x}_{#1}}
\newcommand{\cc}[1]{\tilde{x}_{#1}}
\newcommand{\LB}{\norm{B}}

\newcommand{\Vrob}{V^{\sf RMPC}}

\newcommand{\Vmpcref}{V^{\overline{\sf MPC}}}
\newcommand{\LR}{\norm{R}}

\newcommand{\Lg}{L_g}
\newcommand{\Lpi}{L_{\piref}}

\newcommand{\cE}{\mathcal{E}}
\newcommand{\succE}{\bar{\cE}}
\newcommand{\ccE}{\mathcal{F}}

\newcommand{\xstar}{x^{\star}}

\newcommand{\bdd}[1]{\mathfrak{B}_{#1}}
 
\newcommand{\bddx}{\mathfrak{B}_x}

\newcommand{\ntraj}{\ell}

\newcommand{\Qref}{\bar{Q}}
\newcommand{\Vref}{\bar{J}}
\newcommand{\Vswitch}{\tilde{J}}
\newcommand{\LVref}{L_{\Vref}}
\newcommand{\LVreft}{L_{\Vref_t}}
\newcommand{\LQref}{L_{\Qref}} 
 
\newcommand{\tinf}{{\hat \tau_\infty}}

\newcommand{\barx}{\bar{x}}

\newcommand{\inva}{\Delta_{\W}} 
\newcommand{\invamin}{\Delta_{\W}^{\min}}  
\newcommand{\Xter}{\X_{\sf f}}
\newcommand{\Pter}{P_{\sf f}}
\newcommand{\pter}{p_{\sf f}}

\newcommand{\Olqr}{\mathcal{O}^{\sf lqr}_\infty}

\newcommand{\pstar}{p^\star}
\newcommand{\tstar}{\tau_\infty^\star}

\newcommand{\vu}{\bm u}

\newcommand{\Dcal}{\mathcal{D}}

\newcommand{\bli}{\begin{list}{{\tiny $\blacksquare$}}{\leftmargin=1.5em}
\setlength{\itemsep}{-1pt}
}
\newcommand{\btri}{	\begin{list}{ $\rhd$}{\leftmargin=1.5em}
		\setlength{\itemsep}{-2pt}}
 
\newcommand{\blii}{\begin{list}{{  $\bullet$}}{\leftmargin=0.5em}
\setlength{\itemsep}{-1pt}
}
\newcommand{\eli}{\end{list}}

\definecolor{eblue}{RGB}{0, 50, 180}

\definecolor{egreen}{RGB}{0, 180, 50}
\definecolor{persimmon}{rgb}{0.93, 0.35, 0.0}
\definecolor{darkpastelpurple}{rgb}{0.59, 0.44, 0.84}

\definecolor{MITBrown}{RGB}{164, 31, 50}

\usepackage[utf8]{inputenc} 
\usepackage[T1]{fontenc}    
\usepackage{hyperref}       
\usepackage{url}            
\usepackage{booktabs}       
\usepackage{amsfonts}       
\usepackage{nicefrac}       
\usepackage{microtype}      
\usepackage{xcolor}         
\usepackage{todonotes}

\author{{\bf Kwangjun Ahn} \\ MIT EECS/LIDS \\ \texttt{kjahn@mit.edu} \and {\bf Zakaria Mhammedi} \\ MIT IDSS/LIDS \\ \texttt{mhammedi@mit.edu} \and  {\bf Horia Mania} \\ MIT EECS/LIDS \\ \texttt{hmania@mit.edu} \and {\bf Zhang-Wei Hong} \\ MIT EECS/CSAIL \\ \texttt{zwhong@mit.edu}\and {\bf Ali Jadbabaie} \\ MIT CEE/LIDS/IDSS \\ \texttt{jadbabai@mit.edu} }

\date{\today}

\title{Model Predictive Control via On-Policy Imitation Learning}

\begin{document}

\maketitle
  
\begin{abstract}
In this paper, we leverage the rapid advances in {\it imitation learning}, a topic of intense recent focus in the Reinforcement Learning  (RL) literature, to develop new sample complexity results and performance guarantees for  data-driven Model Predictive Control (MPC) for constrained linear systems. In its simplest form, imitation learning is an approach that tries to learn an expert policy by querying samples from an expert. Recent approaches to  data-driven MPC have used the simplest form of imitation learning  known as behavior cloning to learn  controllers that mimic the performance of MPC by online  sampling of the trajectories of the closed-loop MPC system. Behavior cloning, however, is  a method  that is known to be data inefficient and suffer from distribution shifts. 
As an alternative, we develop a variant of the forward training algorithm which is an on-policy imitation learning method proposed by \citet{ross2010efficient}. 
Our algorithm uses the structure of constrained linear MPC, and our analysis uses the properties of the explicit MPC solution to  theoretically bound the number of online MPC trajectories needed to achieve optimal performance. We validate our results through simulations and show that the forward training algorithm is indeed superior to behavior cloning when applied to MPC.

\end{abstract}
	
\section{Introduction}
Optimization-based control methods such as  model predictive control (MPC) have been among the most versatile techniques in feedback control design for more than 40 years. Such techniques have been  successfully applied to control of dynamic systems in a variety of domains such as autonomous vehicles  \citep{Murray2003,Jad2002,falcone2007predictive,rosolia2017learning}, chemical plants \citep{QIN2003733}, humanoid robots \citep{kuindersma2016optimization}, and many others. Nonetheless, MPC's versatility comes at a  cost. Having to solve optimization problems online makes it difficult to deploy MPC on high-dimensional systems that have strict latency requirements and limited computational or energy resources. To mitigate this issue, considerable effort went into developing faster, tailored optimization methods for MPC \cite{boyd2011distributed,giselsson2013, jerez2014, koegel2011,lucia2016_ARC, lucia2018_TII, mattingley2012cvxgen,richter2012, zometa2013}.

Instead of following these approaches, we pursue a data-driven methodology. We propose and study a scheme to collect data interactively from a dynamical system in feedback with an MPC controller and in order to learn an explicit controller that maps states to inputs. Such approaches are known in the reinforcement learning literature as {\it imitation learning} \citep{pomerleau1988alvinn, schaal1999imitation} and they are well suited for MPC because one can query MPC for the next input at any desired state; all that is needed is to solve the corresponding optimization problem. Nonetheless, in order to learn controllers that are guaranteed to stabilize dynamical systems, to satisfy state and action constraints, and to obtain low cost, we would need to exploit several properties of MPC. 

Our goal of obtaining an explicit map from states to inputs that encapsulates an MPC controller falls under the purview of explicit MPC~\cite{bemporad2002explicit}, which aims to pre-compute and store the solutions of the optimization problems that might be encountered at runtime \cite{alessio2009survey}. 

In general, explicit MPC aims to pre-compute an exact representation of the MPC controller while we aim to learn a controller that performs as well as MPC with high probability. In the same vain, \citet{hertneck2018learning} and \citet{karg2020efficient} suggest learning a controller from data. However, their approaches collect all the trajectory data using MPC before any learning occurs and do not interact with the dynamics further. The lack of interaction in imitation learning is known to lead to sub-optimal performance because  small learning errors would cause a controller produced by such a method to result in states with a different distribution than those produced by MPC during training. In other words, distribution shift leads to error compounding. Our proposed approach completely avoids this issue. To this end, our contributions in this paper can be summarized as follows:

\begin{itemize}
    \item We start by analyzing the imitation learning method known as the forward training algorithm (\forward) in the setting of control affine systems \cite{ross2010efficient}. 
    
    \item 
    We modify \forward{} to make it suitable for MPC applications with constraints. Firstly, \forward{} learns a different controller for each distinct time step and hence it cannot be applied straightforwardly to problems with long or infinite horizons. Fortunately, after sufficiently many times steps, the MPC controller applied to time invariant linear systems becomes equivalent to the classical linear quadratic regulator (LQR) \cite{sznaier1987suboptimal}. We exploit this property; we modify \forward{} to switch to LQR after a number of time steps estimated from data. Secondly, to improve the robustness of our method we require \forward{} to imitate robust MPC \cite{mayne2005robust} instead of standard MPC. We refer to our modified method as \forward-switch.
    
    \item We theoretically guarantee that a controller learned with \forward-switch stabilizes linear systems and satisfies their constraints as long as certain amount of data is available. Moreover, we bound the cost suboptimality of the learned controller, showing that it approaches optimal performance as more data becomes available. None of the previous works on imitating MPC included such guarantees. We also provide theoretical sample complexity bounds using state of the art tools of high dimensional statistics and statistical learning theory.

    \item We validate the efficacy of the modified forward training algorithm on simulated MPC problems, showing that it surpasses non-interactive approaches. 
\end{itemize}

\section{The Forward Training Algorithm for Control}
	\label{sec:forward}
	
In this section, we present the imitation learning method \forward{}~\citep{ross2010efficient} and bound the distance between the trajectories produced by the learned controller and those produced by the expert when the dynamics are control-affine. In subsequent sections, we specialize our analysis to the case where the expert is a MPC controller applied to constrained linear systems.  

Imitation learning aims to learn from demonstrations a controller $\hpi$ that imitates the behavior of a target controller $\expert$, called \emph{expert policy} or simply \emph{expert} in the reinforcement learning literature.
Imitation learning is valuable  when $\expert$ lacks a closed-form expression or is expensive to query in general. For instance, $\expert$ could be a human  performing a task or a MPC controller. More formally, in imitation learning it is assumed that for a state $x$ we can access the input $\expert(x)$. Then, the aim is to use data $\{x_i,\expert(x_i)\}$ to learn a controller $\hpi$ such that $\hpi(x) \approx \expert(x)$.  

In this section, we consider control-affine dynamical systems with constraints: 
\begin{align} \label{def:control_affine}
		x_{t+1} = f(x_t) + g(x_t)u_t\,, \quad x_t \in \X,~u_t \in \U \,,
\end{align}
where $\X\subset \R^{d_x}$ is the state space and $\U\subset \R^{d_u}$ is the input space.
We also find it useful to denote $\varphi_t(x_0, \{u_t\}_{t \geq 0})$ the state $x_t$ that evolves according to $x_{t+1} = f(x_t) + g(x_t)u_t$ and starts at the initial state $x_0 $. 
When the dynamics evolve according to a time-varying feedback controller $\pi = \pi_{0:t-1}$ (i.e.~$\pi_0$ is used at time $0$, $\pi_1$ at time $1$, etc.) we denote the state at time $t$ by $\traj{\pi_{0:t-1}}{x_0}$. If the controller $\pi$ is time-invariant, we simply write $\traj{\pi}{x_0}$. 

Behavior cloning (BC) is the simplest imitation learning method. It consists of collecting $m$ independent trajectories $\traj{\expert}{x_0^{(i)}}$ with initial states $x_0^{(1)}$, $x_0^{(2)}$, \ldots, $x_0^{(m)}$ sampled randomly from an initial distribution $\Dcal$. Then, BC produces a controller $\pibc$ through empirical risk minimization (ERM):
\begin{align} \tag{\sc  Behavior Cloning} \label{BC} 
		\pibc \in \underset{\pi\in\Pi}{\minimize} ~~   \sum_{i=1}^{n} \sum_{t=0}^{T-1} \norm{\pi(\traj{\expert}{x_0^{(i)}})-\expert(\traj{\expert}{x_0^{(i)}})}\,, 
\end{align} 
where $\Pi$ is a class of models that map the state space to the input space and $\norm{\cdot}$ is any norm (although it could be replaced by a more general loss function).
All our results assume that $\expert \in \Pi$. 

\paragraph{Distribution Shift:}The states collected using the expert $\expert$ have a particular distribution $\Dcal^\star$. BC produces a controller $\pibc$ that, when evaluated on samples from $\Dcal^\star$, behaves similarly to the expert $\expert$. However, $\pibc$ is not a perfect copy of the expert and hence the states encountered during its deployment have a different distribution than $\Dcal^\star$. This discrepancy is well known and leads to errors compounding in practice \citep{pomerleau1988alvinn,ross2010efficient, ross2011reduction}. More explicitly, consider an initial state $x_0$ sampled from $\Dcal$. Then, at the first time step $\pibc$ and $\expert$ perform similarly since $\pibc$ was trained using data sampled from $\Dcal$. However, at the second time step the distributions over states produced by $\pibc$ and $\expert$ are different, which means that at the second time step $\pibc$ would be evaluated on a distribution different than the one on which it was trained. Hence, with each time step, $\pibc$ can take the dynamical system to parts of the state space that are less and less covered by the training trajectories resulting in error compounding. 

Since BC does not account for the  intrinsic distribution shift in imitation learning, the number of training trajectories it requires to guarantee a good learned controller can be large (e.g.~exponential in the number of time steps or dimension). The methods for learning a MPC controller due to \citet{hertneck2018learning}, and \citet{karg2020efficient} are variants of behavior cloning and hence also suffer from the presence of distribution shift. Instead, we use and theoretically analyze the forward training algorithm that was initially used by \citet{ross2010efficient} for the tabular MDP setting.
	
\paragraph{Forward Training Algorithm:} \ref{forward} learns a time-varying feedback controller $\hat\pi_{0:T-1}$ in an inductive fashion: during stage $0$, it obtains $\hat \pi_0$ from the ERM \eqref{erm_0}. The controller $\hpi_0$ is used in the dynamical system just at the initial time step. Then, given already learned controllers $\hat \pi_0, \cdots, \hat \pi_{t-1}$, to learn the policy $\hat \pi_t$ for time step $t$, \forward{} samples states $\hat x^{(i)}_t = \varphi_t(x_0^{(i)}; \hat \pi_{0:t-1})$, where $x_0^{(1)},x_0^{(2)},\cdots$ are sampled i.i.d.~from the initial state distribution $\D$.

The advantage of this method is that at time step $t$ during deployment the controller $\hat\pi_t$ would be evaluated on the same distribution as that on which it was trained. Other recent works have also proposed learning inductively time-varying policies as a way to avoid distribution shifts \citep{mhammedi2020learning, sun2019provably}.

	\begin{figure}
		\begin{mdframed}
			\begin{alg}[{\citet{ross2010efficient}}] 
				\namedlabel{forward}{{\forward}} Given $n$ and $T$, a time-varying policy $\hat{\pi}_{0:T-1}$ is computed iteratively according to the following procedure:\\\\ 
	{\bf Stage $0$:}  Sample $n$ initial states $x_0^{(1)} , \cdots, x_0^{(n)}  \sim \mathcal{D}$ and solve the following ERM:
\begin{align}\label{erm_0}
					\hat\pi_0 \in \argmin_{\pi \in \Pi} \ \ \frac{1}{n}\sum_{i=1}^n \|\expert(x_0^{(i)}) -\pi(x_0^{(i)})\|.
				\end{align} 
{\bf Stage $t$:} Sample fresh initial states $x_0^{(1)} , \cdots, x_0^{(n_t)} \sim \mathcal{D}$, where $n_t \coloneqq c n t \lceil \ln^2(t+1)\rceil+n$ and $c\coloneqq \sum_{t=1}^{\infty}1/(t \ln^2 (t+1))$, then evaluate the states $\hat x^{(i)}_t \coloneqq \traj{\hpi_{0:t-1}}{x_0^{(i)}}$, using the  controllers $\hpi_{0:t-1}$ learned in previous stages. Then, select $\hpi_t$ s.t. 	
\begin{align} \label{erm_t}
					\hat\pi_t \in \argmin_{\pi \in \Pi} \ \ \frac{1}{n_t}\sum_{i=1}^{n_t} \|\expert(\hat x^{(i)}_t) -\pi(\hat x^{(i)}_t)\|.
			\end{align}
				
				Since $\expert$ is only defined on $\X$ and since $\hat x^{(i)}_t$ could lie outside $\X$, we define  $\expert(x) = \expert(\proj_\X  x)$.\\
			
{\bf Output:}  The time-varying controller $\hat \pi = \hpi_{0:T-1}$.
			\end{alg}
		\end{mdframed}
	\end{figure}
	 
\subsection{The Sample Complexity of Learning a Controller with \ref{forward}}

In this section, we discuss our statistical guarantees of the controllers produced by \ref{forward}. For simplicity, in this section we consider the setting without state constraints, i.e., $\X = \R^{d_x}$. Before we can state the main results of this section, we need to make an assumption on the class of controllers $\Pi$ used by \forward. 

\begin{assumption}\label{ass:policy}
		The model class $\Pi$ is a finite and contains $\expert$. Moreover, for any $\pi \in \Pi$ and any $x\in \X$ we have $\pi(x)\in \U$. 
	\end{assumption}
	
The second part of the assumption just guarantees that $\Pi$ enforces the input constraints. Any controller class can be modified to satisfy this property by projecting the outputs of the controllers onto $\U$. We assume that the controller class $\Pi$ is finite for simplicity. In this case, our sample complexity guarantees scale with $\ln |\Pi|$---a quantity that arises through a standard generalization bound. When $\Pi$ is not finite, one can replace $\ln |\Pi|$ by learning-theoretic complexity measures such as the Rademacher complexity. Finally, in the MPC application we care about, the assumption $\expert \in \Pi$ is easily satisfied. In the case of constrained linear dynamics with quadratic costs the optimal MPC controller is piecewise affine and it can be expressed as a neural network with ReLU activations as extensively discussed by  \citet[Section I-D]{karg2020efficient} (see also \citep{borrelli2017predictive}).


Now we are ready to state the main result of this section. Its proof relies on the empirical Bernstein inequality \citep{maurer2009empirical} and is deferred to  \autoref{sec:pf_lem_emp_bern}.
	\begin{theorem}
		\label{thm:emp_bern} 
		Let $T\geq 1$ be the target time step, $\delta\in(0,1)$, $n\geq 2$, and $\bdd{u} \coloneqq \sup_{u\in \U} \|u\|$. 
			Let $\hat x_t = \traj{\hpi_{0:t-1}}{x_0}$.
		When \autoref{ass:policy} holds, then under an event $\cE$ of probability at least $1-\delta$ $($over the randomness in the training process$)$, \ref{forward} produces a time-varying controller $\hat\pi_{0:T-1}$ that satisfies 
		\begin{align}
	\mathbb{E} \left[ \left. \|\expert( \hat x_t) -\hat\pi_t( \hat x_t)\| \right| \hat\pi_{0:T-1} \right] \leq   \frac{7 \bdd{u} \ln (2T|\Pi|/\delta)}{n_t}\,, \quad \forall t\geq 0, \label{eq:trainingsuccess0}
		\end{align}
		where $n_t\coloneqq c n t \lceil \ln^2(t+1) \rceil +n$, $c\coloneqq \sum_{t=1}^{\infty}1/(t \ln^2 (t+1))$, and the expectation in \eqref{eq:trainingsuccess0} is with respect to the randomness in the initial state.
	\end{theorem}
	 This result guarantees that the time-varying controller learned by \forward{} is close in expectation to the optimal controller. The following corollary to \autoref{thm:emp_bern} bounds this difference with high probability using Markov's inequality (see \autoref{sec:pf_markov} for a proof):
	\begin{corollary}
	\label{cor:markov}
	Let $\delta\in(0,1)$, $n\geq 2$, and $\bdd{u} \coloneqq \sup_{u\in \U} \|u\|$. 
		Let $\hat x_t = \traj{\hpi_{0:t-1}}{x_0}$.
	When \autoref{ass:policy} holds, then under  the event $\cE$  of probability at least $1-\delta$ $($over the randomness in the training process$)$, \ref{forward} produces a time-varying controller $\hat\pi_{0:T-1}$ such that
	\begin{align}
\label{thm:control_affine}
			\Pr \left[\left. \forall t\geq 0,\ \   \norm{\expert( \hat x_t) - \hat \pi_t ( \hat x_t)} \leq \frac{ 14 \bdd{u}\ln (2T|\Pi|/\delta)}{n\delta} \right| \hat\pi_{0:T-1}\right] \geq 1 -\delta,
\end{align}
where the probability is with respect to the randomness in the initial state.
	\end{corollary}

\paragraph{Infinite Model Classes:}	The results presented in this section assume $\Pi$ is finite for simplicity. This assumption can be easily relaxed. For example, to get an analogue of the result of \autoref{thm:emp_bern} for a infinite class $\Pi$, one can use the empirical Bernstein inequality \citep[Lemma 6]{maurer2009empirical}, which replaces $\ln |\Pi|$ by the logarithm of a ``growth function'' for the class $\Pi$ (see  \autoref{app:nonfinite}). In the case where $\Pi$ is a class of ReLU Neural Networks, the latter quantity can be bounded by $\widetilde O(N_{\sf params})$, where $N_{\sf params}$ is the number of parameters of the Neural Networks in $\Pi$.

\paragraph{Trajectory Guarantees:} \autoref{thm:emp_bern} guarantees that \forward{} produces a controller $\hpi$ that generates inputs to the system that are close to those outputted by $\expert$. However, this result does not immediately imply that $\hpi$ and $\expert$ follow similar trajectories  (errors could compound over time, causing $\hpi$'s trajectories to diverge from those of $\expert$).
Following the main ideas of \citet{tu2021sample} and \citet{pfrommer2022taylor}, one can in fact show guarantees in terms of trajectories when the closed-loop system under $\expert$ is robust in an appropriate sense. 
See  \autoref{sec:thm_traj} the details.

In subsequent sections, we refine the results presented so far to the case of MPC. 
 
\section{Background on the Control of Linear Systems}	 \label{intro:mpc}
	 
In this section, we review some background material on the control of linear systems, with a focus on the linear quadratic regulator (LQR) and on MPC. This section is not intended to be exhaustive; we only cover the notions needed in this work. We consider the  linear dynamical system
\begin{align} \label{def:linear}
x_{t+1} =A x_t + B u_t, 
\end{align}
 which clearly maps to the control-affine setting \eqref{def:control_affine} with $f(x_t) = Ax_t$ and $g(x_t)=B$.

\paragraph{Unconstrained Optimal Control:} Suppose we desire to solve the following infinite horizon optimal control problem
\begin{align}  \tag{\sf LQR}\label{eq:lqr}
		\begin{split}
			&\underset{ u_0,u_1,u_2\cdots}{\text{minimize}}  ~~  \sum_{t=0}^{\infty}{x_t^\top Q x_t + u_t^\top R u_t}   \\
			&\text{subject to}  ~~x_{t+1} = Ax_t + Bu_t,
		\end{split}
\end{align}  
where $R \succ 0$ and $Q \succeq 0$. Then, the optimal controller can be computed in closed form. More specifically,  let $\plqr$ be the unique positive definite solution to the discrete algebraic Riccati equation 
	\begin{align}\label{pare}
		P= A^\top  P A - A^\top P B(R + B^\top 
		PB)^{-1} B^\top P A+ Q\,.
	\end{align}
Then, the optimal controller for \eqref{eq:lqr} is given by
\begin{align}\label{dare}
		\dlqr(x)= \klqr x, \quad \text{ where}~\klqr = - (R + B^\top \plqr B)^{-1} B^\top \plqr A\,.
\end{align}
One important property of $\dlqr$ is that the closed-loop system induced by the controller is stable. In other words, if $A_{\klqr}\coloneqq A+B\klqr$, we have $\rho(A_{\klqr})<1$.

\paragraph{Constrained Linear Dynamics and MPC:} Suppose we wish to design a controller for the dynamics \eqref{def:linear} such that $x_t \in \X$ and $u_t \in \U$ for all $t \geq 0$ and suppose we still wish to minimize the quadratic cost shown in \eqref{eq:lqr}. However, solving 
an infinite horizon problem under the constraints 
$x_t \in \X$, $u_t \in \U$ is computationally challenging. Moreover, it is not sufficient to find an optimal sequence of inputs $\{u_t\}$ because open-loop control is brittle in the presence of noise. 

MPC precisely resolves these issues by designing a {\bf feedback controller} using the following {\bf finite horizon} $N$-step  optimal control problem.
For any given initial state $x \in\X$ and a sequence of control inputs $ \vu = (u_0,u_1,\cdots,u_{N-1})$, 
\begin{align} \label{def:VN}
    V_N(x, \vu) : = \sum_{k=0}^{N-1}{x_k^\top Q x_k + u_k^\top R u_k}  +  x_N^\top \Pter x_N\,,\quad \text{s.t.}~x_{t+1} = Ax_t + Bu_t,~x_0=x,
\end{align}   
where $\Pter, Q \succeq 0$ and $R \succ 0$.
Then, the finite horizon problem is:
\begin{align}  \tag{\sf  MPC}\label{eq:mpc}
	 \underset{ \vu = (u_0,u_1,\cdots,u_{N-1})}{\text{minimize}}  \left\{ V_N(x,\vu) ~|~ x_t \in \X,u_t \in \U~\forall t,~\text{and}~ x_N\in \Xter\right\}\,,
\end{align} 
where $\X,\U,\Xter$ are constraint sets that are closed and contain the origin.

When the constraint sets are convex the finite-horizon problem is a convex optimization problem that can be solved efficiently.
Let $\ustar_0(x), \ustar_1(x), \dots, \ustar_{N-1}(x)$ be the optimal solution to \eqref{eq:mpc}.
To obtain a feedback controller, instead of deploying all inputs $\ustar_0(x), \ustar_1(x), \dots, \ustar_{N-1}(x)$, \ref{eq:mpc} only deploys the first input $\ustar_0(x)$. Then, it observes the next state of the system and solves another $N$-step problem starting at the new state.
In particular, given the initial state $x_0$, the MPC controller is defined as $\mpc(x_0)\coloneqq\ustar_0(x_0)$. After controlling the system for a single step using $\ustar_0(x_0)$, and say that the next state is $x_1$, \ref{eq:mpc} then resolves another $N$-step finite horizon problem starting from $x_1$ and use $\ustar_0(x_1)$ and so on.
Since the  window of time over which the finite horizon problem is solved is shifting to the right by one at each step, \ref{eq:mpc} is also referred to as \emph{receding horizon control (RHC}). We refer readers to textbooks
(e.g., \citet{morari1999model,rawlings2017model,borrelli2017predictive}) for extensive background.

MPC is a popular and successful control strategy because it can systematically handle multi-input-multi-output systems, nonlinearities, as well as constraints. The main drawback of MPC is that MPC must solve an optimization problem at each time step.
For this reason, traditional applications were limited to slow systems such as chemical processes~\citep{qin2003,rawlings2009}.

 Now let us discuss the feasible domain of the optimization problem \eqref{eq:mpc}. See, e.g., \citep[Ch. 10-12]{borrelli2017predictive} for more context and details. 
\begin{definition}[Feasible Domain of \ref{eq:mpc}] \label{def:feasible}
Let $\X$, $\U$ be constraint sets and $\Xter$ be a terminal set and recursively define the set $\X_N,\X_{N-1},\dots, \X_0$ as $\X_N = \Xter$, and  $\X_i \coloneqq \{x \in \X~:~ \exists u \in\U~\st~Ax+Bu\in \X_{i+1} \}$, for $i=N-1,N-2,\dots, 0$.   
Then, $\X_0$ is the feasible domain of \eqref{eq:mpc} w.r.t.~$(\X,\U,\Xter)$.
\end{definition}

\paragraph{Persistent Feasibility and Stability:}
Ensuring the optimization problems \eqref{eq:mpc} are feasible at each time step and ensuring that MPC stabilizes the underlying dynamics requires careful arguments. Merely having $x_0 \in \X_0$ does not necessarily imply persistent feasibility, and a careful choice of terminal cost $\Pter$ and terminal constraint  $\Xter$ has to be made. Here, persistent feasibility means that if $x_0 \in \X_0$ then $\varphi_t(x_0;\mpc)\in \X_0$, for all $t\geq 1$.

A sufficient condition for  persistent feasibility is to choose $\Xter$ as a control invariant set \citep[Theorem   12.1]{borrelli2017predictive}, and for concreteness, we consider the set
that is invariant with respect to the LQR controller $\dlqr$ as below. 
 \begin{definition}[Positive Invariance w.r.t.~LQR Controller] \label{def:pos_inv}
 Let $\dlqr(x)=\klqr x$ denote the unconstrained LQR controller. We say that a set $\mathcal{O}$ is \emph{positively invariant} with respect to $(\X,\U)$, if $\mathcal{O} \subseteq \X$ and whenever $x_0\in \mathcal{O}$, $x_t\in \mathcal{O}$ and $\dlqr(x_t)=\klqr x_t \in \U$, for all $t\geq 0$, where $x_{t+1} = (A+B\klqr) x_t$.
    Let $\Olqr(\X,\U)$ be the \emph{maximal positively invariant} set with respect to $(\X,\U)$.
\end{definition}
\noindent As detailed in \citet[Sec.~10.2]{borrelli2017predictive}, the maximal positively invariant set $\Olqr(\X,\U)$ (or its polytopic inner approximations) can be computed using polytopic computations.

Finally, we discuss a sufficient condition for the stability of the dynamics in feedback with MPC.  It is well known that the MPC controller  stabilizes the system (i.e.~$\norm{x_t}\to 0$ as $t\to \infty$) if it uses a Control Lyapunov Function (CLF) $\pter$ as the terminal cost, where $\pter$ is a CLF if 
\begin{align} \label{cond:control_lya}
   \min_{u\in \U,~ Ax+Bu\in\Xter } \left(\pter(Ax+Bu)-\pter(x)+\ell(x,u)\right)\leq 0,\quad  \forall x\in \Xter,
\end{align} 
where $\ell(x,u) = x^\top Qx+u^\top R u$ denotes the stage cost. The MPC objective in \eqref{def:VN} has the function $\pter(x)=x^\top \Pter x$ as the terminal cost, which can be made to fulfill \eqref{cond:control_lya} by choosing $\Pter = \plqr$ and $\Xter=\Olqr(\X,\U)$.
With these choices, we also have the following useful property that we use later:
\begin{align} \label{lqr_tinf}
    \mpc(x) =\dlqr(x)\quad\text{whenever}~x\in \Olqr\,.
\end{align}
See, e.g., \cite[Sec.~2.5.4]{rawlings2017model} for details. At a high level, when $x\in \Olqr$, the inputs produced by the LQR controller $\dlqr$ correspond to the optimal solution of \eqref{eq:mpc} since they are the optimal solution to the unconstrained objective $V_N$ thanks to the choice $\Pter=\plqr$ and the positive invariance of $\Olqr$---see \autoref{def:pos_inv}.
Now, we are ready to discuss the main method and result of this work. 

\section{On-policy Imitation Learning for MPC}
\label{sec:forward-switch}

In this section, we discuss how to adapt \forward{} to imitate MPC and we offer refined guarantees on the performance of the modified \forward{} method. We consider the dynamics, where unlike in \autoref{sec:forward}, we allow $\X\subsetneq\R^{d_x}$:
	\begin{align*}
		x_{t+1}\coloneqq Ax_{t} + Bu_{t} \,, \quad x_t \in \X,~u_t \in \U \,. 
	\end{align*}
All imitation learning methods need a choice of model class $\Pi$. We follow \citet{karg2020efficient} and choose $\Pi$ to be a class of neural networks with ReLU activations. This choice is appropriate because MPC implements a piecewise affine controller with the different pieces supported on polytopic regions when the constraints on the dynamics are polytopes. The main challenge in precomputing the piecewise affine controller implied by MPC is that the number of polytopic regions is exponential in the horizon $N$ and other problem dependent terms. However, polynomially many parameters are sufficient in order to express the MPC controller as a ReLU NN \cite{karg2020efficient}.

In order for \ref{forward} to imitate MPC efficiently, we modify it in the following ways:
\bli
\item {\bf Robust MPC as the Expert Policy:}  A guarantee of the form $\norm{\expert( \hat x_t) - \hat \pi_t ( \hat x_t)} \leq \eps$ does not necessarily ensure that the learned policy stabilizes the system: $\eps$-deviations from the expert at each iteration can compound and lead to instability. To mitigate this issue, we use robust MPC \emph{\`{a} la} \citet{mayne2005robust} as the expert, which we detail in \autoref{sec:robust_MPC}. \citet{hertneck2018learning} also chose robust MPC as the expert, but their method is a form of behavior cloning that requires an extra validation step. 

\item {\bf Sample-efficient Implementation:} \forward{} learns a time-varying controller, which allows it to elude the challenge of distribution shift. However, as explained by \citet{ross2010efficient}, learning a time-varying controller implies that the  sample complexity grows with the number of stages $T$. Therefore, with a straightforward application of \forward{} it would not be possible to stabilize a dynamical system over an infinite horizon. To address this drawback, we use an insight of \citet{sznaier1987suboptimal}. Namely, after using MPC for a certain a number of time steps the state of the dynamics reaches a region on which MPC and the infinite horizon LQR agree. Therefore, our version of \forward{} estimates the number of time steps to switch to the time-invariant LQR controller. 
\eli	
 
\noindent We refer to the modified method as  \ref{forward-switch}, and we present its performance guarantees. 

	\subsection{Robust MPC as the Expert Policy}
	\label{sec:robust_MPC}
Before discussing the theoretical guarantees of our learned controller, we first review  the robust MPC method that we use as the expert $\expert$. Although we consider noiseless dynamics, it is useful to introduce disturbances in order to account for the errors introduced by the learned controller. Robust MPC is a controller that is robust against disturbances $w\in \W$ at each step, where $\W$ is a compact set. The robust MPC controller proposed by \citet{mayne2005robust} differs from standard MPC in two ways: 
	\bli
		\item robust MPC shrinks the constraint sets $\X$ and $\U$ in order to account for the disturbances,
		\item robust MPC takes the first input produced by the MPC optimization problem and it linearly interpolates it with a stabilizing controller:  
	\eli
Before we can discuss the details of robust MPC, we introduce the following notion: 
	\begin{definition}[Disturbance Invariant Set {\citep{kolmanovsky1998theory}}] \label{def:disturbance_inva}
		Given a compact disturbance set $\W$, we say that $\inva$ is a disturbance invariant set if it is a neighborhood around the origin that satisfies $A_{\klqr}\inva +\W \subseteq \inva$, where $A_{\klqr}\coloneqq A+B\klqr$ (recall that $\rho(A_{\klqr})<1$). 
	\end{definition}
\noindent\citet[Section 4]{kolmanovsky1998theory} show that the minimal disturbance invariant set is \begin{align} \label{eq:min_inv}
		\invamin \coloneqq \sum_{k=0}^\infty A_{\klqr}^k \W \,.
	\end{align}
In our case, we let $\W$ be the ball of radius $\eps>0$ centered at the origin, i.e.~$\W= \B(\eps)$.
	Below, we estimate the radius of $\invamin$ based on the fact that $A_{\klqr}$ is stable.
	Since $A_{\klqr}$ is stable, there exists $\rho \in (\rho(A_{\klqr}),1)$ and $\tau>0$ such that $\norm{A_{\klqr}^k} \leq \tau \cdot \rho^k$ for all $k$ (see, e.g., \citep[eq (3)]{mania2019certainty}).
	
	\begin{claim} \label{claim:inva}
		For $\eps>0$ let $\cW=\B(\eps)$.
		Then,  $\invamin \subseteq   \B (\kappa \cdot\eps)$ with $\kappa\coloneqq \frac{\tau}{1-\rho}$.
	\end{claim}
	\begin{proof}
		From the fact $\norm{A_{\klqr}}^k \leq \tau \cdot \rho^k$ for all $k$, it follows that $\norm{A_{\klqr}^k \W } \leq \tau \rho^k \eps$ for all $k$.
		This implies that $A_{\klqr}^k \W \subset \B(\tau \rho^k \eps)$.
		Thus,
		$\invamin \subseteq \sum_{k=0}^\infty \B(\tau \rho^k \eps) = \B(\frac{\tau}{1-\rho}\cdot \eps)$.  
	\end{proof}
\noindent In light of  \autoref{claim:inva}, here and below  readers can consider 
\begin{align}
    \W = \B(\eps) \quad\text{and}\quad \inva=\B(\kappa\cdot \eps).
\end{align}
Now, for a given state $\xinit$, let us consider \ref{eq:mpc} with slightly stricter constraints.
We begin with a notation: for two arbitrary sets $\X$ and $\mathcal{Y}$, $\X \ominus \mathcal{Y}$ is defined as $\mathcal{X}\ominus \mathcal{Y} \coloneqq\{x \in \X \,|\,x+\mathcal{Y}\subseteq \mathcal{X}\}$.
With this notation, we consider the following  constraint sets, terminal set, and the positive invariant set:
\bli
\item $\Xref \coloneqq \X \ominus \inva$ and $\Uref \coloneqq  \U\ominus \klqr \inva$.
\item $\Xterref$ is chosen as $\Olqrref\coloneqq\Olqr(\Xref,\Uref)$, where $\Olqr(\Xref,\Uref)$ is the maximal positive invariant set with respect to $(\Xref, \Uref)$---see \autoref{def:pos_inv}. 
\eli 
 Using the new constraints and the terminal set, consider 
\begin{align}   \tag{$\overline{\sf  MPC}$}\label{eq:mpcref}
	 \underset{ 
	 \vu=(u_0,\dots,u_{N-1})}{\text{minimize}}  \left\{ V_N(x,\vu) ~|~ x_t \in \Xref  , ~ u_t \in\Uref~,\forall t,~\text{and}~ x_N\in \Xterref \right\}\,.
\end{align} 
where $V_N$ is defined in \eqref{def:VN}.   
The set of initial states $\Xzeroref$ for which \ref{eq:mpcref} admits a solution is the feasible domain of \ref{eq:mpc} with respect to the constraint sets $(\Xref,\Uref, \Xterref)$---see \autoref{def:feasible}. 
Let $\piref$ be the MPC controller defined by \ref{eq:mpcref}; that is, for each $x \in \Xzeroref$, $\piref(x)$ is given by 
\begin{align*}
    \piref(x) \coloneqq   \uref_0(x),\quad \text{where $(\uref_0(x),\uref_1(x),\dots, \uref_{N-1}(x))$ is the optimal solution to \ref{eq:mpcref}.}
\end{align*}
Henceforth, we assume that the initial state distribution $\D$ is almost surely supported within the feasibility set $\Xzeroref$.

Then, the key idea of \citet{mayne2005robust} is to include the initial point $x_0$  as a parameter of the optimization problem: given $x \in (\X_0 \oplus \inva) \cap \X$,
	\begin{align}   \tag{{\sf RMPC}} \label{eq:mpc_robust}
	\underset{ 
	 x_0,\vu }{\text{minimize}}  \left\{ V_N(x_0,\vu) ~|~ \text{constraints of \ref{eq:mpcref} and }  x \in x_0 \oplus \inva \right\}\,. 
\end{align} 
Letting $\xref_0(x), \uref_0(x),\uref_1(x),\dots, \uref_{N-1}(x)$ be the optimal solution to \ref{eq:mpc_robust}, the robust MPC controller of \citet{mayne2005robust} is defined as  
\begin{align}
    \rmpc (x) \coloneqq \uref_0(x)+ \klqr (x_0- \xref_0(x)) =  \piref(\xref_0(x))+ \klqr (x_0- \xref_0(x))\,.
\end{align}
The following result establishes a key property of the robust MPC controller $\rmpc$. 
We include the proof in \autoref{pf:mayne} for completeness.
\begin{proposition}[{\citep[Proposition 3 and Theorem 1]{mayne2005robust}}]\label{thm:mayne}
For any $\xinit \in (\X_0 \oplus \inva) \cap \X$, the robust MPC controller $\rmpc$ robustly stabilizes the system with disturbances
\begin{align} \label{exp:sys_disturb}
    	x_{t+1}\coloneqq Ax_{t} + Bu_{t} +w_t\,, \quad x_0=\xinit,~x_t \in \X,~u_t \in \U,~ w_t\in \W \,, 
\end{align}
in the sense that there exists a constant $\zeta\in(0,1)$ such that $\norm{\xref_0(x_t)} = O(\zeta^t \norm{\xref_0(\xinit)})$ for all $t\in \N$. 
\end{proposition}

\begin{remark}[Time Step to Reach Positive Invariance] \label{rmk:tstar}
Note that the conclusion $\norm{\xref_0(x_t)} = O(\zeta^t \norm{\xref_0(\xinit)})$ holds for any choices of disturbances $\{w_t\}$ as long as $w_t \in \W$, for all $t\in[T]$. 
Hence, for such disturbances, as long as the support of $\D$ is almost surely bounded, there must exist $\tstar$ such that $\xref_0(x_{\tstar}) \in \Olqrref \ominus \inva$. 
Note that $\tstar$ depends solely on the system parameters and can be regarded as an absolute constant.
Then, since $x_{\tstar} \in \xref_0(x_{\tstar})\oplus \inva$, it follows that $\forall \{w_t\}\in \W$, $x_{\tstar} \in \Olqrref$.   
 \end{remark}
Next, we use the robust MPC and propose an efficient implementation of \forward{}.

\subsection{\ref{forward-switch}: Efficient Application of \ref{forward} to  MPC}
	\label{sec:main_MPC}

\begin{figure}
\begin{mdframed}
			\begin{alg2}  
		\namedlabel{forward-switch}{\algmpc}  
$T=$ Imitation learning times steps. $\tinf$ is initialized as $\tinf=T$. We use \ref{forward} with $\expert$ chosen as the robust MPC controller $\rmpc$ to learn $\tmpc_{0:T-1}$ as per the following procedure:  \\

{\bf Forward training until positive invariance:} At the end of each stage of \ref{forward}, say the $(t-1)$-th stage, we sample $\ell$ trajectories according to our learned controller $\hpi_{0:t-1}$ to generate $\hat x^{(i)}_{t}$, $i=1,2,\dots,\ell$.
\begin{list}{{\tiny $\blacksquare$}}{\leftmargin=1.0em}
\setlength{\itemsep}{-1pt}
    \item[-] If $\hat x^{(i)}_{t}\in \Olqrref$ for all $i=1,2,\dots,\ell$, then we terminate \forward{} early and set $\tinf=t$.
    \item[-] Otherwise, proceed to the next stage. 
\end{list}
 
{\bf Output policy:} Output a time-varying policy $\tmpc_{0:T-1}$ defined as
\begin{align}   \label{switch}
		\tmpc_t \coloneqq \left\{\begin{array}{ll} \hmpc_t, & \text{if }t <   \tinf,\\ 
		\dlqr, & \text{if }\tinf\leq t\leq T-1   \,,  \end{array}  \right.
	\end{align} 
where $\dlqr$ is the unconstrained infinite horizon LQR controller defined in \eqref{dare}. 
\end{alg2}
\end{mdframed}
	\end{figure} 
We are finally ready to formally present the modified \forward{} method---\ref{forward-switch}. 
As explained by \citet[Section 3]{ross2010efficient}, the main limitation of \ref{forward} is that the number of stages increases with the horizon length $T$. 
We modify the method so that with high probability we only need $\tstar$ stages, where $\tstar$ is the time step required for \ref{eq:mpc_robust} to reach $\Olqrref$ (see \autoref{rmk:tstar}) that is independent of $T$. The main idea is that after at most $\tstar$ steps, the robust exponential stability of the MPC controller ensures that states enter the positively invariant set $\Olqrref$ under any sequence of $\eps$-bounded disturbances. Whenever the state enters the positively invariant set $\Olqrref$, due to the choice of the terminal cost $\Pter = \plqr$, we know that the MPC controller $\piref$ coincides with $\dlqr$, which can be computed explicitly and stored efficiently, and so there is nothing more to learn. This insight, which goes back to \citet{sznaier1987suboptimal}, was in fact already used in the control literature to come up with an efficient algorithm for computing the constrained LQR controller~\citep{chmielewski1996constrained,scokaert1998constrained,grieder2004computation}. 

To implement this idea, \ref{forward-switch} must first estimate the number of steps the learned controller $\hmpc$ requires to drive the state to $\Olqrref$.
Given that \ref{forward} learns the controller incrementally for each time step, one can estimate the number of steps $t$ by checking if all the states at stage $t$ from the generated trajectories have reached $\Olqrref$. 
The next theorem justifies the step \eqref{switch} of  \ref{forward-switch} that switches the learned controller $\hmpc_t$ to $\dlqr$ for $t\geq \tinf$: we show that with high probability, it holds that $\tinf \leq \tstar$ and $\hat x_{\tinf}$ lies in $\Olqrref$ (in which the expert policy is indeed $\dlqr$). The proof of the next theorem relies on \autoref{thm:mayne} and \autoref{cor:markov}. The full details can be found in \autoref{pf:stability}.  
\begin{theorem}	\label{thm:stability}
	Let $\delta, \varepsilon\in(0,1)$ and $\bdd{u} \coloneqq \sup_{u\in \U} \|u\|$. Suppose  \autoref{ass:policy} holds and that the support of $\D$ is almost surely bounded. Choose $n\geq \frac{14\ell\norm{B}\bdd{u} \ln (2T\ell |\Pi|/\delta)}{\eps \delta}$ and $\ntraj$ such that $\ntraj \geq \frac{10\ln (T/\delta)}{\delta}$.
Then, under an event  $\succE$ of probability at least $1-3\delta$ $($over the randomness in the training process$)$, the stopping time $\tinf$ in \ref{forward-switch} satisfies $\tinf \leq \tstar$ ($\tstar$ defined in \autoref{rmk:tstar}) and
\begin{subequations}
\label{eq:gar}
\begin{gather}
 \Pr[\hat x_\tinf \in \Olqrref\mid \tinf, \hmpc]\geq 1 -\delta,  \\
\Pr\left[\left. \forall t = 0,\dots,\tinf-1, \   \norm{\rmpc( \hat x_t) - \hmpc_t ( \hat x_t)} \leq \eps/\|B\|\ \right|\tinf, \hmpc  \right]\geq 1- \delta,  \label{eq:specific}
\end{gather}
\end{subequations}
where the probabilities are over the randomness in the initial state. Further, under $\succE$ and the events in \eqref{eq:gar}, the controller $\tmpc_{0:T-1}$ does not violate any constraints.
\end{theorem}
\begin{remark}[Sample Complexity of \ref{forward-switch}]
Note that under the setting of \autoref{thm:stability}, the total number of expert demonstrations required by \ref{forward-switch} is upper bounded by $\widetilde O(n \tinf)=\widetilde O(\frac{\ell \tinf }{\eps \delta})=\widetilde O(\frac{\tinf}{\eps \delta^2})$, where $\widetilde O$ hides polylog factors in $T$, $\delta$, and $|\Pi|$. Since \autoref{thm:stability} guarantees $\tinf \leq \tstar$, the total number of expert demonstrations is thus upper bounded by $\widetilde O(\frac{\tstar \wedge T}{\epsilon \delta^2})$. Crucially, for large enough imitation learning horizon $T$ (in particular, for $T\geq \tstar$), the number of required trajectories depends only logarithmically on the horizon $T$ (since $\tstar$ is to be treated as a system's constant independent of $T$---see \autoref{rmk:tstar}).
\end{remark}
 We are left to quantify the cost achieved by the learned controller, which we do next. 

	\subsection{Performance Guarantees} 
	\label{sec:theory_MPC}
In this subsection, we bound the suboptimality of the controller $\tilde \pi$ learned by \ref{forward-switch}. For the theoretical analysis, we first define the Q-function of the reference controller $\piref$ defined in \ref{eq:mpcref}. Let $\Vref_t:\Xref_0 \to \R$ be the $t$-step cost function of $\piref$, i.e., for $x_0\in \Xref_0$ and   $\vv{s} \coloneqq \trajj{\piref}{x_0}{s}$, 
\begin{align*}
    \Vref_t(x_0) \coloneqq \sum_{t=0}^{t-1} \ell(\vv{s}, \piref(\vv{s}))\,.  
\end{align*}
where $\ell(x,u) = x^\top Qx+u^\top R u$ denotes the stage cost.
For $x\in \Xref_0$, and an input $u \in \Uref$ such that $Ax+Bu\in \Xref_0$, define 
	\begin{align}
		\Qref_{t}(x, u) \coloneqq \ell(x,u)+ \Vref_{T-t-1}(Ax + B u)\,. \label{eq:Qstar}
	\end{align}  
Let $\Vswitch_t$ be the $t$-step cost of the learned policy $\tmpc$ from \ref{forward-switch} (defined in a similar way to $\Vref_t$).  
We now state the performance guarantee of \ref{forward-switch}.

	\begin{theorem}[Performance Guarantee] \label{thm:performance}
	
  	Let $\delta, \varepsilon\in(0,1)$ and assume the same conditions as \autoref{thm:stability}.  
	Then, under the same event $\succE$ as \autoref{thm:stability} and for any $x_0$ satisfying the events in \eqref{eq:gar}, we have
		\begin{align}
		 \Vswitch_T(x_0) - \Vref_T(x_0)  \leq O( \tstar \eps) \,,
		\end{align}
where $O(\cdot)$ hides an absolute constant that depends on the system parameters and $\tstar$ is a system's constant independent of $T$---see \autoref{rmk:tstar}.
\end{theorem}
The proof of this result is deferred to \autoref{pf:performance}.

	\section{Experiments}
	\label{sec:experiment}

In this section, we demonstrate our theoretical results for the MPC application through a set of experiments.
We demonstrate that \BC{} can indeed suffer from distribution shift and destabilize the system,
while \forward{} can cope with this issue.

\begin{figure}
    \centering
    \includegraphics[width=0.7\textwidth]{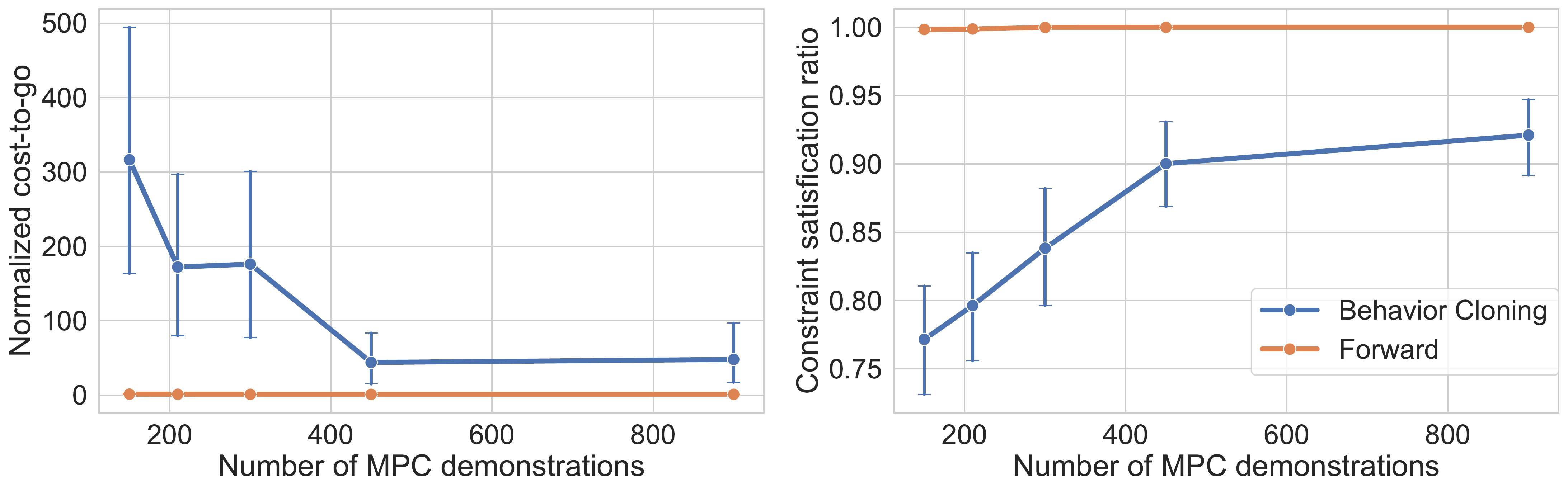}
     \includegraphics[width=0.7\textwidth]{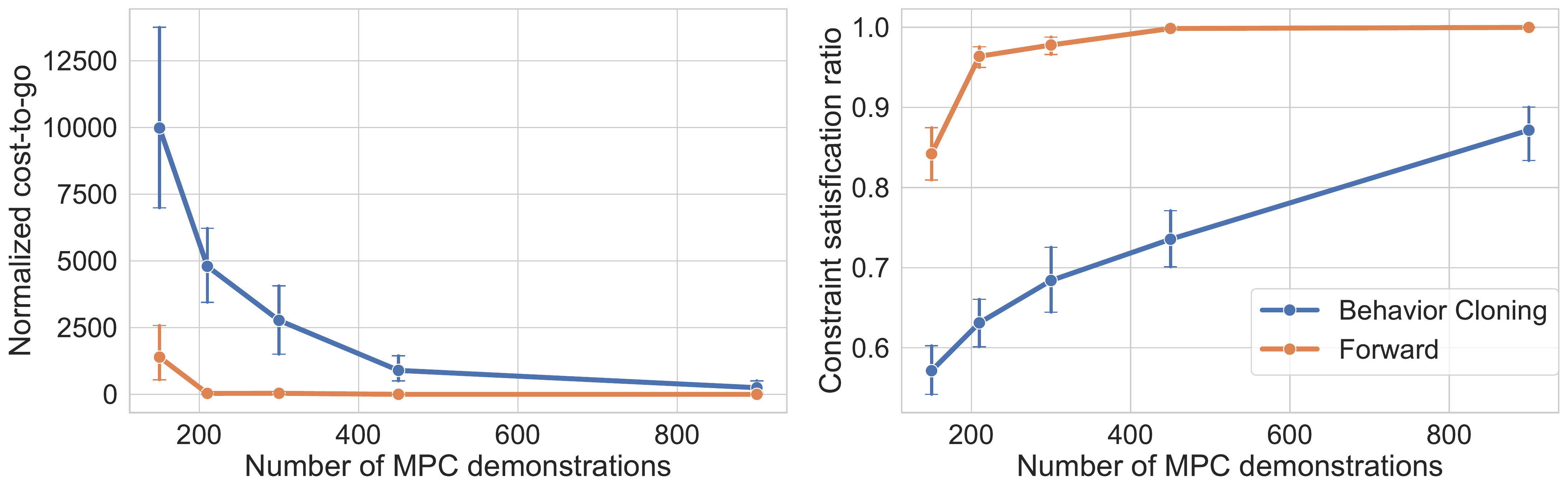}
    \caption{The results for the normalized cost-to-go and the constraint satisfaction ratio for the trajectory length $T=30$.  The first row contains the results for $d=3$, and the second row shows the results for $d=5$.}
    \label{fig:unstable}
\end{figure} 
\paragraph{Experimental Setup.} For $d\in \{3,5\}$, we consider an open-loop unstable dynamical system $x_{t+1} = Ax_t + Bu_t$, where $A\in \R^{d\times d}$ is chosen as an upper triangular matrix whose diagonal entries are $1.1$ and the upper diagonal entries are chosen from the uniform distribution over $[-2,2]$ (see \autoref{app:matrix} for the $A$ matrices used for the plots), and $B\in \R^{d\times 1}$ is chosen as $[0 \ 0 \ \cdots \ 1]^\top$. 
We impose the constraints $x_t \in [-100,100]^d$, $u_t\in [-10,10]$ and choose the initial state distribution $\D$ as the uniform distribution over $[8,10]^d$.
We set the horizon $N$ of MPC to be $20$ and the number of imitation learning time steps $T$ to be $30$, and we use  pyMPC~\cite{forgione2020efficient} for implementing MPC demonstrations. 
In the MPC optimization, we did not impose the terminal constraint.

To parametrize the policies we use a fully connected neural network with three hidden layers. Each layer has $50$ neurons followed by ReLU activations.
For optimization, we use the Adam optimizer with a learning rate of $0.001$. We train the policies for $500$ epochs.

\paragraph{Results.} 
We first compare the performance of \BC{} and \forward{}. For each algorithm, we measure the normalized cost $J^{\sf algorithm}(x_0)/J^{\sf mpc}(x_0)$ for $20$ different test initial states $x_0$ sampled from $\D$. 
Moreover, we report the constraint satisfaction ratio along the test trajectories. 
We repeat each setting in the experiment for $50$ times and report the $95\%$ confidence intervals with error bars.
The results are reported in \autoref{fig:unstable}.
As one can see from \autoref{fig:unstable}, for these systems, there is a significant difference in performance between the two algorithms.
For $d=3$,  the mean normalized cost of \forward{} is less than $1.2$ for all settings, while that of \BC{} is greater than $40$ even with $900$ MPC demonstrations.
For $d=5$, the normalized cost-to-go of \forward{} is less than $1.13$ with $450$ MPC demonstrations, while that of \BC{} is higher than $240$ even with $900$ MPC demonstrations.

\begin{figure}
    \centering
  \includegraphics[width=0.35\textwidth]{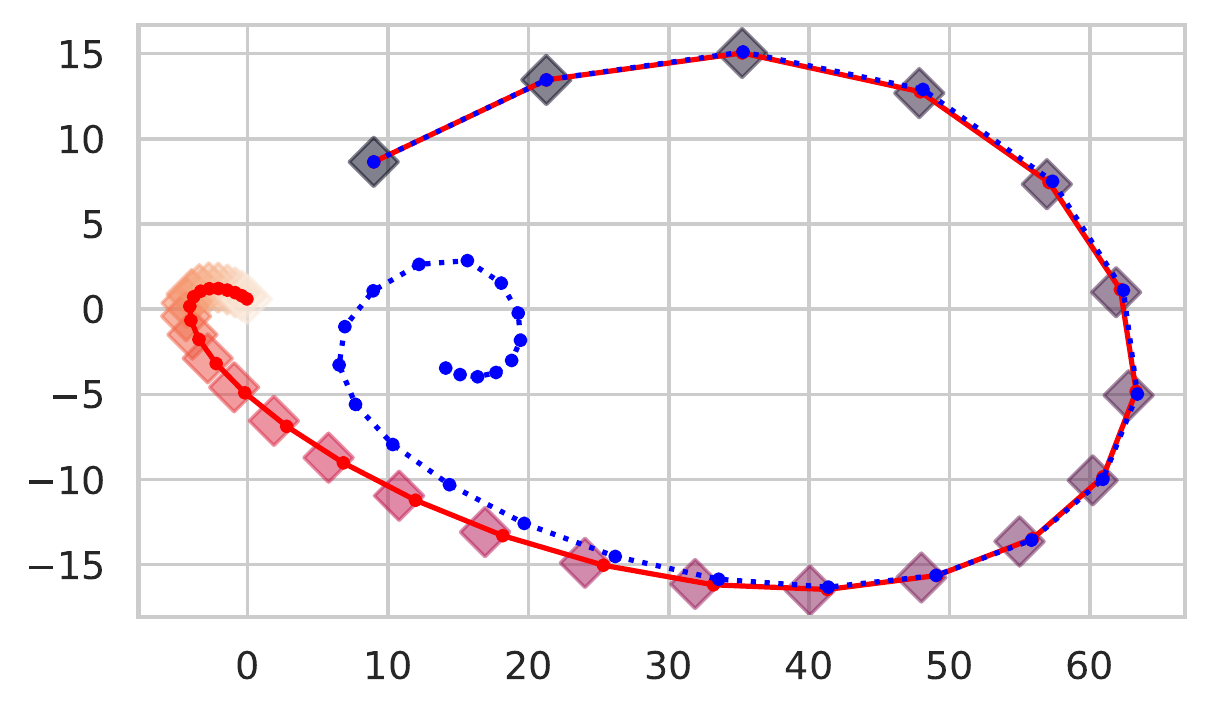}
  \includegraphics[width=0.35\textwidth]{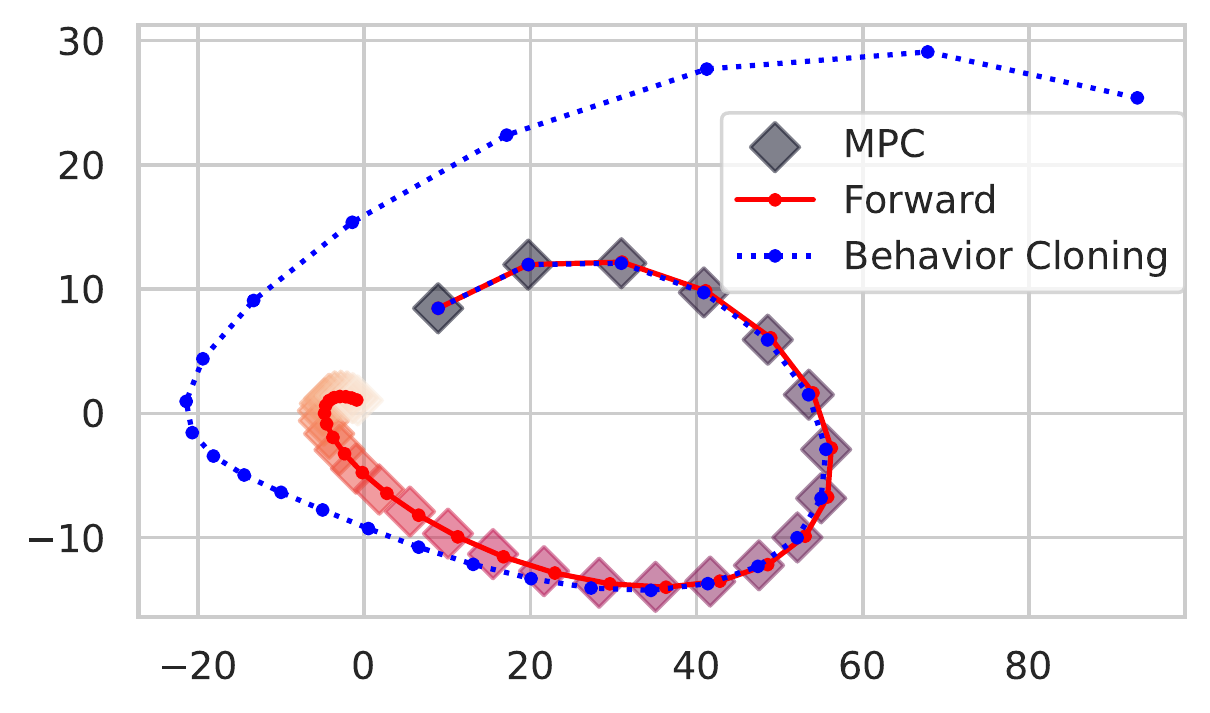} 
      \caption{The  first two coordinates of the sample trajectories produced by MPC, \forward{}, \BC{}.  The left plot is  for $d=3$, and the right column is for $d=5$.}
    \label{fig:unstable_traj}
\end{figure} 

 In order to visualize the results, we plot the first two coordinates of the sample trajectories produced by each controller  in \autoref{fig:unstable_traj}. As one can see from the figure, \BC{} indeed suffers from the distribution shift issue: small errors in the learned controller pile up along time steps and lead the trajectory to a region where the learned controller cannot stabilize the system.

\begin{remark}[Results for 2D systems]
 We also tried several $2$-dimensional systems, including (i) the $d=2$ case of our simulated system and (ii) double integrators (especially,  the versions in \cite[Section VI-A]{chen2018approximating} and \cite[Section VI-A]{jones2010polytopic}).
Interestingly, for these systems we tried, we did not see much difference in the performance between the two algorithms.
\end{remark}
 
We now test the performance of our proposed method \ref{forward-switch}.
For the same systems as before, we estimate $\tinf$ as per the procedure described in \ref{forward-switch}, where we check if the sample trajectory $\hat x^{(i)}_{t}$ lies in a subset\footnote{For ease of implementation, we use the subset that is defined by the level set of the terminal cost,  i.e., $\pter\coloneqq x^\top \plqr x$. It is well-known that the level set of $\pter$ is positive invariant w.r.t.~the LQR Controller. We simply choose the maximal level set of $\pter\coloneqq x^\top \plqr x$ in which the constraints are not violated under the LQR controller.} of $\Olqrref$.
Our estimated $\tinf$ for the $d=5$ case is $12$.

In \autoref{fig:switch}, we report the mean normalized cost-to-go and the constraint satisfaction ratio of    \ref{forward-switch}.
Notably, \ref{forward-switch} achieves the mean normalized cost-to-go of $\approx 1.034$ with only $180$ MPC demonstrations, while  \forward{} achieves the mean normalized cost-to-go of $\approx 35$ when trained using $210$ MPC demonstrations. 
Hence, our experiment indicates that \ref{forward-switch} is indeed more sample-efficient in some situations.

 We also compare the performance of  \ref{forward-switch} with its \BC{} counterpart.
For a fair comparison, we also train  \BC{} for $T=12$ steps and then for time steps greater than $12$, we employ the LQR controller.
In \autoref{fig:switch}, we report the mean normalized cost-to-go and the constraint satisfaction ratio of    \ref{forward-switch} and its \BC{} counterpart.
Although $T=12$ is smaller than the previous experiment setting where $T=30$,  we still see a noticeable difference in the performance between the two algorithms. 

\begin{figure}
    \centering
    \includegraphics[width=0.7\textwidth]{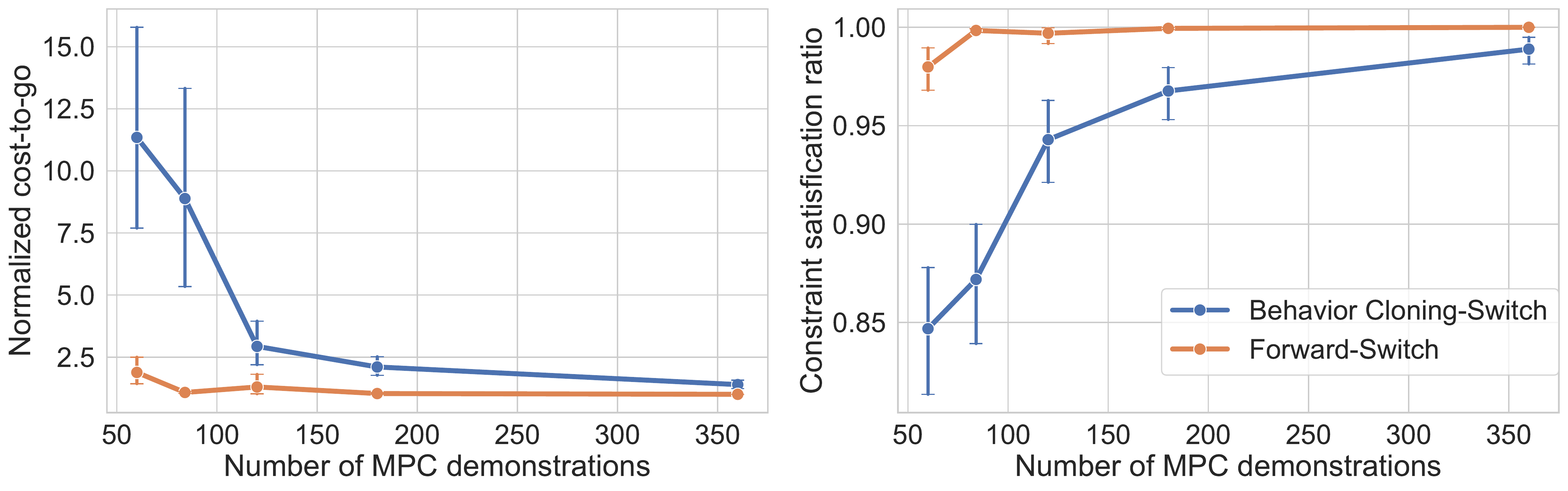}
    \caption{The performance comparison between \ref{forward-switch} and its \BC{} counterpart. The result is for $d=5$. Here the estimated number of steps $\tinf$ to reach the positive invariant set $\Olqrref$ is  $12$.
    }
    \label{fig:switch}
\end{figure}

\section{Conclusion}

 In this work, we leverage techniques from imitation learning to circumvent MPC's reliance on online optimization. More specifically, we adapt an interactive imitation learning algorithm called the forward training algorithm to take advantage of MPC's properties. When presented with a constrained linear system we show that our modified method learns a controller that stabilizes the dynamics, satisfies the state and input constraints, and achieves cost as good as that obtained by MPC.  We validate our results through simulations and compare the modified forward training algorithm with other data-driven methods.

We conclude this paper with interesting future directions. 
An alternative approach to ours is to learn the value function instead of the policy. In particular, it is known that the MPC value function is convex and piecewise  quadratic  \citep{bemporad2002explicit,seron2003characterisation}.
It might be interesting to see whether such properties make the approach based on learning value functions more desirable.
More broadly, whether  the value of each expert demonstration can be used to improve performance of imitation learning algorithms would be of great interest.
Lastly, combining our  approach with a direct policy optimization approach (e.g.,  \citet{chen2018approximating}) would be of great practical interest, given that a direct policy optimization typically requires more samples.

\section*{Acknowledgement}

Kwangjun Ahn, Zakaria Mhammedi, Horia Mania and Ali Jadbabaie were supported by the ONR grant (N00014-20-1-2394) and MIT-IBM Watson as well as a Vannevar Bush fellowship from Office of the Secretary of Defense.	
Zakaria Mhammedi was also supported by the ONR grant (N00014-20-1-2336).
Zhang-Wei Hong was supported by the ONR MURI grant (N00014-22-1-2740).
Kwangjun Ahn also acknowledges support from the Kwanjeong Educational Foundation.

Part of this work was done as Kwangjun Ahn's class project for 6.832: Underactuated Robotics at MIT, Spring 2022; Kwangjun Ahn thanks  Russ Tedrake for constructive comments on the project.
The authors thank Jack Umenberger and Sinho Chewi for very detailed comments regarding the theoretical results in the paper.
The authors also thank Haoyuan Sun and Navid Azizan for fruitful discussions during the initial stage of this work.
	
	\bibliographystyle{plainnat}
	\bibliography{ref}

\begin{thebibliography}{47}
\providecommand{\natexlab}[1]{#1}
\providecommand{\url}[1]{\texttt{#1}}
\expandafter\ifx\csname urlstyle\endcsname\relax
  \providecommand{\doi}[1]{doi: #1}\else
  \providecommand{\doi}{doi: \begingroup \urlstyle{rm}\Url}\fi

\bibitem[Alessio and Bemporad(2009)]{alessio2009survey}
Alessandro Alessio and Alberto Bemporad.
\newblock A survey on explicit model predictive control.
\newblock In \emph{Nonlinear model predictive control}, pages 345--369.
  Springer, 2009.

\bibitem[Bemporad et~al.(2002)Bemporad, Morari, Dua, and
  Pistikopoulos]{bemporad2002explicit}
Alberto Bemporad, Manfred Morari, Vivek Dua, and Efstratios~N Pistikopoulos.
\newblock The explicit linear quadratic regulator for constrained systems.
\newblock \emph{Automatica}, 38\penalty0 (1):\penalty0 3--20, 2002.

\bibitem[Borrelli et~al.(2017)Borrelli, Bemporad, and
  Morari]{borrelli2017predictive}
Francesco Borrelli, Alberto Bemporad, and Manfred Morari.
\newblock \emph{Predictive control for linear and hybrid systems}.
\newblock Cambridge University Press, 2017.

\bibitem[Boyd et~al.(2011)Boyd, Parikh, Chu, Peleato, and
  Eckstein]{boyd2011distributed}
Stephen Boyd, Neal Parikh, Eric Chu, Borja Peleato, and Jonathan Eckstein.
\newblock Distributed optimization and statistical learning via the alternating
  direction method of multipliers.
\newblock \emph{Foundations and Trends{\textregistered} in Machine Learning},
  3\penalty0 (1):\penalty0 1--122, 2011.

\bibitem[Chen et~al.(2018)Chen, Saulnier, Atanasov, Lee, Kumar, Pappas, and
  Morari]{chen2018approximating}
Steven Chen, Kelsey Saulnier, Nikolay Atanasov, Daniel~D Lee, Vijay Kumar,
  George~J Pappas, and Manfred Morari.
\newblock Approximating explicit model predictive control using constrained
  neural networks.
\newblock In \emph{American Control Conference}, pages 1520--1527. IEEE, 2018.

\bibitem[Chmielewski and Manousiouthakis(1996)]{chmielewski1996constrained}
Donald Chmielewski and V~Manousiouthakis.
\newblock On constrained infinite-time linear quadratic optimal control.
\newblock \emph{Systems \& Control Letters}, 29\penalty0 (3):\penalty0
  121--129, 1996.

\bibitem[Falcone et~al.(2007)Falcone, Borrelli, Asgari, Tseng, and
  Hrovat]{falcone2007predictive}
Paolo Falcone, Francesco Borrelli, Jahan Asgari, Hongtei~Eric Tseng, and Davor
  Hrovat.
\newblock Predictive active steering control for autonomous vehicle systems.
\newblock \emph{IEEE Transactions on control systems technology}, 15\penalty0
  (3):\penalty0 566--580, 2007.

\bibitem[Forgione et~al.(2020)Forgione, Piga, and
  Bemporad]{forgione2020efficient}
Marco Forgione, Dario Piga, and Alberto Bemporad.
\newblock Efficient calibration of embedded {MPC}.
\newblock In \emph{Proc. of the 21st IFAC World Congress 2020, Berlin, Germany,
  July 12-17 2020}, 2020.

\bibitem[Foster and Simchowitz(2020)]{foster2020logarithmic}
Dylan Foster and Max Simchowitz.
\newblock Logarithmic regret for adversarial online control.
\newblock In \emph{International Conference on Machine Learning}, pages
  3211--3221. PMLR, 2020.

\bibitem[Giselsson et~al.(2013)Giselsson, Doan, Keviczky, Schutter, and
  Rantzer]{giselsson2013}
Pontus Giselsson, Minh~Dang Doan, Tamás Keviczky, Bart~De Schutter, and Anders
  Rantzer.
\newblock Accelerated gradient methods and dual decomposition in distributed
  model predictive control.
\newblock \emph{Automatica}, 49\penalty0 (3):\penalty0 829 -- 833, 2013.
\newblock ISSN 0005-1098.
\newblock \doi{https://doi.org/10.1016/j.automatica.2013.01.009}.

\bibitem[Grieder et~al.(2004)Grieder, Borrelli, Torrisi, and
  Morari]{grieder2004computation}
Pascal Grieder, Francesco Borrelli, Fabio Torrisi, and Manfred Morari.
\newblock Computation of the constrained infinite time linear quadratic
  regulator.
\newblock \emph{Automatica}, 40\penalty0 (4):\penalty0 701--708, 2004.

\bibitem[Hertneck et~al.(2018)Hertneck, K{\"o}hler, Trimpe, and
  Allg{\"o}wer]{hertneck2018learning}
Michael Hertneck, Johannes K{\"o}hler, Sebastian Trimpe, and Frank
  Allg{\"o}wer.
\newblock Learning an approximate model predictive controller with guarantees.
\newblock \emph{IEEE Control Systems Letters}, 2\penalty0 (3):\penalty0
  543--548, 2018.

\bibitem[Jadbabaie and Hauser(2002)]{Jad2002}
Ali Jadbabaie and John Hauser.
\newblock Control of a thrust-vectored flying wing: a receding horizon—lpv
  approach.
\newblock \emph{International Journal of Robust and Nonlinear Control},
  12\penalty0 (9):\penalty0 869--896, 2002.

\bibitem[Jerez et~al.(2014)Jerez, Goulart, Richter, Constantinides, Kerrigan,
  and Morari]{jerez2014}
Juan~L Jerez, Paul~J Goulart, Stefan Richter, George~a Constantinides, Eric~C
  Kerrigan, and Manfred Morari.
\newblock {Embedded Online Optimization for Model Predictive Control at
  Megahertz Rates}.
\newblock \emph{IEEE Transactions on Automatic Control}, 59\penalty0
  (12):\penalty0 3238--3251, 2014.
\newblock ISSN 0018-9286.
\newblock \doi{10.1109/TAC.2014.2351991}.

\bibitem[Jones and Morari(2010)]{jones2010polytopic}
Colin~N Jones and Manfred Morari.
\newblock Polytopic approximation of explicit model predictive controllers.
\newblock \emph{IEEE Transactions on Automatic Control}, 55\penalty0
  (11):\penalty0 2542--2553, 2010.

\bibitem[Karg and Lucia(2020)]{karg2020efficient}
Benjamin Karg and Sergio Lucia.
\newblock Efficient representation and approximation of model predictive
  control laws via deep learning.
\newblock \emph{IEEE Transactions on Cybernetics}, 50\penalty0 (9):\penalty0
  3866--3878, 2020.

\bibitem[K{\"o}gel and Findeisen(2011)]{koegel2011}
M.~K{\"o}gel and R.~Findeisen.
\newblock {A fast gradient method for embedded linear predictive control}.
\newblock In \emph{{Proceedings of the 18th IFAC World Congress}}, pages
  1362--1367, 2011.

\bibitem[Kolmanovsky and Gilbert(1998)]{kolmanovsky1998theory}
Ilya Kolmanovsky and Elmer~G Gilbert.
\newblock Theory and computation of disturbance invariant sets for
  discrete-time linear systems.
\newblock \emph{Mathematical problems in engineering}, 4\penalty0 (4):\penalty0
  317--367, 1998.

\bibitem[Kuindersma et~al.(2016)Kuindersma, Deits, Fallon, Valenzuela, Dai,
  Permenter, Koolen, Marion, and Tedrake]{kuindersma2016optimization}
Scott Kuindersma, Robin Deits, Maurice Fallon, Andr{\'e}s Valenzuela, Hongkai
  Dai, Frank Permenter, Twan Koolen, Pat Marion, and Russ Tedrake.
\newblock Optimization-based locomotion planning, estimation, and control
  design for the atlas humanoid robot.
\newblock \emph{Autonomous robots}, 40\penalty0 (3):\penalty0 429--455, 2016.

\bibitem[Lucia et~al.(2016)Lucia, K\"ogel, Zometa, Quevedo, and
  Findeisen]{lucia2016_ARC}
S~Lucia, M.~K\"ogel, P.~Zometa, D.~E. Quevedo, and R.~Findeisen.
\newblock Predictive control, embedded cyberphysical systems and systems of
  systems -- {A} perspective.
\newblock \emph{Annual Reviews in Control}, 41:\penalty0 193--207, 2016.

\bibitem[Lucia et~al.(2018)Lucia, Navarro, Lucia, Zometa, and
  Findeisen]{lucia2018_TII}
S~Lucia, D.~Navarro, O.~Lucia, P.~Zometa, and R.~Findeisen.
\newblock Optimized {FPGA} implementation of model predictive control using
  high level synthesis tools.
\newblock \emph{IEEE Transactions on Industrial Informatics}, 14\penalty0
  (1):\penalty0 137--145, 2018.

\bibitem[Mania et~al.(2019)Mania, Tu, and Recht]{mania2019certainty}
Horia Mania, Stephen Tu, and Benjamin Recht.
\newblock Certainty equivalence is efficient for linear quadratic control.
\newblock \emph{Advances in Neural Information Processing Systems}, 32, 2019.

\bibitem[Mattingley and Boyd(2012)]{mattingley2012cvxgen}
Jacob Mattingley and Stephen Boyd.
\newblock {CVXGEN:} a code generator for embedded convex optimization.
\newblock \emph{Optimization and Engineering}, 13\penalty0 (1):\penalty0 1--27,
  2012.

\bibitem[Maurer and Pontil(2009)]{maurer2009empirical}
Andreas Maurer and Massimiliano Pontil.
\newblock Empirical {Bernstein} bounds and sample variance penalization.
\newblock In \emph{Conference on Learning Theory, {COLT} 2009}, Montreal,
  Canada, 18--21 June 2009.

\bibitem[Mayne(2001)]{mayne2001control}
David~Q Mayne.
\newblock Control of constrained dynamic systems.
\newblock \emph{European Journal of Control}, 7\penalty0 (2-3):\penalty0
  87--99, 2001.

\bibitem[Mayne and Langson(2001)]{mayne2001robustifying}
David~Q Mayne and Wilbur Langson.
\newblock Robustifying model predictive control of constrained linear systems.
\newblock \emph{Electronics Letters}, 37\penalty0 (23):\penalty0 1422--1423,
  2001.

\bibitem[Mayne et~al.(2005)Mayne, Seron, and Rakovi{\'c}]{mayne2005robust}
David~Q Mayne, Mar{\'\i}a~M Seron, and SV~Rakovi{\'c}.
\newblock Robust model predictive control of constrained linear systems with
  bounded disturbances.
\newblock \emph{Automatica}, 41\penalty0 (2):\penalty0 219--224, 2005.

\bibitem[Mhammedi et~al.(2020)Mhammedi, Foster, Simchowitz, Misra, Sun,
  Krishnamurthy, Rakhlin, and Langford]{mhammedi2020learning}
Zakaria Mhammedi, Dylan~J Foster, Max Simchowitz, Dipendra Misra, Wen Sun,
  Akshay Krishnamurthy, Alexander Rakhlin, and John Langford.
\newblock Learning the linear quadratic regulator from nonlinear observations.
\newblock \emph{Advances in Neural Information Processing Systems},
  33:\penalty0 14532--14543, 2020.

\bibitem[Morari and Lee(1999)]{morari1999model}
Manfred Morari and Jay~H Lee.
\newblock Model predictive control: past, present and future.
\newblock \emph{Computers \& Chemical Engineering}, 23\penalty0 (4-5):\penalty0
  667--682, 1999.

\bibitem[Murray et~al.(2003)Murray, Hauser, Jadbabaie, Milam, Petit, Dunbar,
  and Franz]{Murray2003}
Richard~M Murray, John Hauser, Ali Jadbabaie, Mark~B Milam, Nicolas Petit,
  William~B Dunbar, and Ryan Franz.
\newblock Online control customization via optimization-based control.
\newblock In \emph{Software-Enabled Control, Information technology for
  dynamical systems}, pages 149--174. Wiley Online Library, 2003.

\bibitem[Pfrommer et~al.(2022)Pfrommer, Zhang, Tu, and
  Matni]{pfrommer2022taylor}
Daniel Pfrommer, Thomas~T.C.K. Zhang, Stephen Tu, and Nikolai Matni.
\newblock {TaSIL}: Taylor series imitation learning.
\newblock \emph{arXiv preprint arXiv:2205.14812}, 2022.

\bibitem[Pomerleau(1988)]{pomerleau1988alvinn}
Dean~A Pomerleau.
\newblock Alvinn: An autonomous land vehicle in a neural network.
\newblock \emph{Advances in neural information processing systems}, 1, 1988.

\bibitem[Qin and Badgwell(2003{\natexlab{a}})]{QIN2003733}
S.Joe Qin and Thomas~A. Badgwell.
\newblock A survey of industrial model predictive control technology.
\newblock \emph{Control Engineering Practice}, 11\penalty0 (7):\penalty0
  733--764, 2003{\natexlab{a}}.

\bibitem[Qin and Badgwell(2003{\natexlab{b}})]{qin2003}
S.Joe Qin and Thomas~A. Badgwell.
\newblock A survey of industrial model predictive control technology.
\newblock \emph{Control Engineering Practice}, 11:\penalty0 733--764,
  2003{\natexlab{b}}.

\bibitem[Rawlings et~al.(2017)Rawlings, Mayne, and Diehl]{rawlings2017model}
James~Blake Rawlings, David~Q Mayne, and Moritz Diehl.
\newblock \emph{Model predictive control: theory, computation, and design},
  volume~2.
\newblock Nob Hill Publishing Madison, WI, 2017.

\bibitem[Rawlings and Mayne(2009)]{rawlings2009}
J.B. Rawlings and D.Q. Mayne.
\newblock \emph{Model Predictive Control Theory and Design}.
\newblock Nob Hill Pub, 2009.

\bibitem[Richter et~al.(2012)Richter, Jones, and Morari]{richter2012}
Stefan Richter, Colin~Neil Jones, and Manfred Morari.
\newblock {Computational complexity certification for real-time MPC with input
  constraints based on the fast gradient method}.
\newblock \emph{IEEE Transactions on Automatic Control}, 57\penalty0
  (6):\penalty0 1391--1403, 2012.
\newblock ISSN 00189286.
\newblock \doi{10.1109/TAC.2011.2176389}.

\bibitem[Rosolia and Borrelli(2017)]{rosolia2017learning}
Ugo Rosolia and Francesco Borrelli.
\newblock Learning model predictive control for iterative tasks. a data-driven
  control framework.
\newblock \emph{IEEE Transactions on Automatic Control}, 63\penalty0
  (7):\penalty0 1883--1896, 2017.

\bibitem[Ross and Bagnell(2010)]{ross2010efficient}
St{\'e}phane Ross and Drew Bagnell.
\newblock Efficient reductions for imitation learning.
\newblock In \emph{Proceedings of the thirteenth international conference on
  artificial intelligence and statistics}, pages 661--668. JMLR Workshop and
  Conference Proceedings, 2010.

\bibitem[Ross et~al.(2011)Ross, Gordon, and Bagnell]{ross2011reduction}
St{\'e}phane Ross, Geoffrey Gordon, and Drew Bagnell.
\newblock A reduction of imitation learning and structured prediction to
  no-regret online learning.
\newblock In \emph{Proceedings of the fourteenth international conference on
  artificial intelligence and statistics}, pages 627--635. JMLR Workshop and
  Conference Proceedings, 2011.

\bibitem[Schaal(1999)]{schaal1999imitation}
Stefan Schaal.
\newblock Is imitation learning the route to humanoid robots?
\newblock \emph{Trends in cognitive sciences}, 3\penalty0 (6):\penalty0
  233--242, 1999.

\bibitem[Scokaert and Rawlings(1998)]{scokaert1998constrained}
Pierre~OM Scokaert and James~B Rawlings.
\newblock Constrained linear quadratic regulation.
\newblock \emph{IEEE Transactions on automatic control}, 43\penalty0
  (8):\penalty0 1163--1169, 1998.

\bibitem[Seron et~al.(2003)Seron, Goodwin, and
  De~Don{\'a}]{seron2003characterisation}
Maria~M Seron, Graham~C Goodwin, and Jos{\'e}~A De~Don{\'a}.
\newblock Characterisation of receding horizon control for constrained linear
  systems.
\newblock \emph{Asian Journal of Control}, 5\penalty0 (2):\penalty0 271--286,
  2003.

\bibitem[Sun et~al.(2019)Sun, Vemula, Boots, and Bagnell]{sun2019provably}
Wen Sun, Anirudh Vemula, Byron Boots, and Drew Bagnell.
\newblock Provably efficient imitation learning from observation alone.
\newblock In \emph{International conference on machine learning}, pages
  6036--6045. PMLR, 2019.

\bibitem[Sznaier and Damborg(1987)]{sznaier1987suboptimal}
Mario Sznaier and Mark~J Damborg.
\newblock Suboptimal control of linear systems with state and control
  inequality constraints.
\newblock In \emph{26th IEEE conference on decision and control}, volume~26,
  pages 761--762. IEEE, 1987.

\bibitem[Tu et~al.(2022)Tu, Robey, Zhang, and Matni]{tu2021sample}
Stephen Tu, Alexander Robey, Tingnan Zhang, and Nikolai Matni.
\newblock On the sample complexity of stability constrained imitation learning.
\newblock \emph{Proceedings of the 4rd Conference on Learning for Dynamics and
  Control, arXiv preprint arXiv:2102.09161}, 2022.

\bibitem[Zometa et~al.(2013)Zometa, K\"ogel, and Findeisen]{zometa2013}
P.~Zometa, M.~K\"ogel, and R.~Findeisen.
\newblock {$\mu$AO-MPC}: A free code generation tool for embedded real-time
  linear model predictive control.
\newblock In \emph{Proceedings of the American Control Conference}, pages
  5320--5325, June 2013.

\end{thebibliography}


	\appendix

	 \clearpage

\section{A Generalization Guarantee for Non-Finite Policy Classes}
	\label{app:nonfinite}
 
To present the generalization guarantee of \ref{forward} when the policy class $\Pi$ is non-finite, we need to define the notation of growth function:
\begin{definition}[Growth function~{\citep[pg. 2]{maurer2009empirical}}]
For $\eps >0$, a function class $\mathcal{F}= \{f:X\to \R\}$ and an integer $n$, the growth function $\mathcal{N}_{\infty}\left(\eps,\mathcal{F},n\right)$ is defined as 
\begin{equation*}
\mathcal{N}_{\infty }\left( \eps ,\mathcal{F},n\right) =\sup_{\mathbf{x}\in X^{n}}\mathcal{N}\left( \eps ,\mathcal{F}\left(\mathbf{x}\right),\left\Vert \cdot\right\Vert _{\infty }\right),\end{equation*}
where $\mathcal{F}\left( \mathbf{x}\right) =\left\{ \left( f\left(x_{1}\right) ,\cdots,f\left( x_{n}\right) \right) :f\in \mathcal{F}\right\}\subseteq \mathbb{R}^{n}$ and for $A\subseteq \mathbb{R}^{n}$ the number $\mathcal{N}\left( \eps ,A,\left\Vert \cdot\right\Vert_{\infty }\right)$ is the smallest cardinality $\left\vert A_{0}\right\vert $ of a set $A_{0}\subseteq A$ such that $A$ is contained in the union of $\eps$-balls centered at points in $A_{0}$, in the metric induced by $\left\Vert\cdot\right\Vert_{\infty}$.
\end{definition} 
We now state an analogue of \autoref{thm:emp_bern} when $\Pi$ is non-finite:
\begin{lemma}
	\label{thm:emp_bern_nonfinite}
	Let $\delta\in(0,1)$, $n\geq 16$, $\mathcal{M}(n)\coloneqq 10 \mathcal{N}_{\infty}(1/n, \Pi, 2n)$, and $\bdd{u} \coloneqq \sup_{u\in \U} \|u\|$. Further, $\hat \pi$ be the time-varying policy \ref{forward} and suppose that Assumption \ref{ass:policy} holds. Then, with probability at least $1-T\delta$, we have 
	\begin{align}
		\forall t\in[0,T-1],\quad 
		\mathbb{E} \left[\|\pi_\star(\hat x_t) -\hat\pi_t(\hat x_t)\| \right] \leq  \frac{30 \bdd{u} \ln (\mathcal{M}(n)/\delta)}{n-1}. \label{eq:error}
	\end{align}
\end{lemma}
\begin{proof}
	Let $\ell_t(x)\coloneqq \|\pi_\star(x) -\hat \pi_t(x)\|$.
	By the Empirical Bernstein Inequality \citep[Theorem 6]{maurer2009empirical}, we have,
	with probability at least $1-\delta$, 
	\begin{align}
		\E_{x \sim \calD_t} \ell_t(\hat x_t) -\frac{1}{n}\sum_{i=1}^{n}\ell_t(\hat x_t^{(i)}) \leq   \sqrt{\frac{18 \widehat V_{t,n} \ln (\mathcal{M}(n)/\delta) }{n}} + \frac{30  \bdd{u} \ln (\mathcal{M}(n)/\delta)}{n-1},\label{eq:prefact}
	\end{align}
	where $\widehat V_{t,n} \coloneqq \frac{1}{n-1}\sum_{i=1}^{n} (\ell_t(\hat x_t^{(i)})- \frac{1}{n}\sum_{j=1}^{n}\ell_t(\hat x_t^{(i)}))^2$ is the empirical loss variance. Since $\pi_\star$ is in $\Pi$, we must necessarily have that $\frac{1}{n}\sum_{i=1}^{n}\ell_t(\hat x_t^{(i)})=0$. Furthermore, since $\widehat V_{t,n} \leq  \frac{4 \bdd{u}}{n}\sum_{i=1}^{n}\ell_t(\hat x_t^{(i)})$, we must also have $\widehat V_{t,n}=0$. Plugging these facts in \eqref{eq:prefact} implies the desired result.
\end{proof}

\section{Trajectory Guarantees under Robustness of Expert Controller}
\label{sec:thm_traj}

\citet{tu2021sample} show that one can prove the guarantees in terms of trajectories if the closed-loop system produced by $\expert$ is robust in the following sense.

	\begin{definition}[{Incremental Input-to-state Stability}] \label{def:ids}
		Let $\W$ be a compact subset of $\R^{d_x}$. For disturbances $w_t\in \W$, consider the discrete-time dynamics $x_{t+1} = f(x_t) + w_t$.
		Let $a,b\geq 1$, $\gamma >0$.
		We say that $f$ is $(a,b,\gamma)${\em-incrementally input-to-state stable ($\delta$ISS)}  if
		for all $T \in \N$, initial state $x_0=x_0\in \X$,
		\begin{align*} 
			\sum_{t=0}^{T} \norm{\traj{\{w_t\}}{x_0}  - \traj{\{0\}}{x_0}}^{a}  \leq  
			\gamma \cdot  {\sum_{t=0}^{T-1}}  \norm{w_t}^{b}\,.
		\end{align*} 
	\end{definition}
	See, e.g.,  \citep[Section 3.1]{tu2021sample} for examples.
	Now,  we demonstrate that our guarantee \eqref{thm:control_affine} translates into a trajectory guarantee under the $(a,b,\gamma)$-$\delta$ISS condition on the expert closed-loop system. The proof is deferred to  \autoref{sec:pf_thm_traj}.
	
	\begin{theorem}
		\label{thm:traj}
		Let $\delta\in(0,1)$. Consider the control-affine discrete-time dynamics $x_{t+1}=\fclexp(x_t)\coloneqq f(x_t) +g(x_t)\expert(x_t)$ with $\norm{g}_{\infty}\leq \Lg$, and suppose that $\fclexp$ is $(a,b,\gamma)$-$\delta$ISS with $b \leq a$. 
		Let $\hpi$ be the controller outputted by \forward{}. Also,  let $\xstar_t$ be the trajectory produced by $\expert$ and $\hat x_t$ be the trajectory produced by $\hpi$.
		Then, when \autoref{ass:policy} holds and $n, T \geq 2$, with probability at least $1-\delta$ (over the randomness in the training process), the time-varying controller $\hpi$ satisfies
		\begin{align*}
			\E\left[\left. \frac{1}{T}\sum_{t=0}^{T} \norm{\xstar_t- \hat x_t } \right| \pihat \right]  \leq  \left(\frac{\gamma}{T}\right)^{1/a} \cdot  \left(\frac{14\Lg \bdd{u}\ln (2T|\Pi|/\delta)}{n}\right)^{b/a}.
		\end{align*}
	\end{theorem}
	Being an on-policy imitation learning algorithm, our guarantee does not require the ERM problems to be constrained unlike the result of \citet{tu2021sample}. More specifically, the version of BC used in \citep[Algorithm 2]{tu2021sample} imposes $\pibc$ to be incrementally input-to-state stable by solving a constrained ERM problem. 
However, \citet{tu2021sample} acknowledged that there is no efficient way to enforce such a constraint in practice.
\begin{remark}
A very recent work by \citet{pfrommer2022taylor} also develops an algorithm proposed  that does not require solving a constrained ERM problem. 
However, their approach is quite different than ours. 
Roughly speaking, they additionally assume that the derivative of expert controller $\nabla \expert$ is  continuous, and then add additional terms to the ERM objective that forces $\nabla \hpi$ to be close to $\nabla \expert$. 
This requirement does not make their approach directly applicable to our main MPC application as the derivative of the MPC controller may not be continuous. It would be interesting to see if one can adapt their approach to the MPC application.
\end{remark}

\subsection{Proof of \autoref{thm:traj}} 
\label{sec:pf_thm_traj}
	
	To prove the theorem, we first recall the following result from \citet{tu2021sample} that can be proved using  H{\"{o}}lder's inequality and Jensen's inequality.
	
	\begin{proposition}[{\citep[Proposition C.4]{tu2021sample}}] \label{prop:ddi}
		Suppose that $f$ is $(a,b,\gamma)$-$\delta$ISS with $b \leq a$. Then, for all $T \in \N$,   $x_0 \in \X$ and random disturbance sequences $\{w_t\}_{t \geq 0} \subseteq \W$, we have
		\begin{align*}
		 \sum_{t=0}^{T} 	 \norm{\traj{ \{w_t\} }{x_0} -\traj{ \{0\}}{x_0}}  &\leq \gamma^{{1}/{a}} T^{1-{1}/{a}} 
	\left(\sum_{t=0}^{T-1} \norm{w_t} \right)^{{b}/{a}} .
		\end{align*}
	\end{proposition}

    Now consider $\fclexp$ as the underlying dynamics and let $\trajs{t}{\{w_t\}}{x_0}$ be the trajectory with respect to $\fclexp$ starting from $x_0$ with disturbances $\{w_t\}$. 
		Then, we have $\trajs{t}{\{0\}}{x_0} = \traj{\expert}{x_0}$.
		Here, $\traj{\expert}{x_0}$ is the trajectory w.r.t the original underlying dynamics  \eqref{def:control_affine}.
		\begin{claim} \label{claim:traj}
			Then for the disturbances $\dis_t \coloneqq g(\hat x_t)(\hpi_t(\hat x_t)- \expert(\hat x_t))$, we have  $\trajs{t}{\{\dis_i\}}{x_0} = \hat x_t$.
		\end{claim}
		\begin{proof}[Proof of \autoref{claim:traj}.]		We prove the claim by induction on $t$.
			First consider the base case. Then $\trajs{1}{\dis_0}{x_0}= \fclexp (x_0) + g(x_0)(\hat\pi(x_0)-\expert(x_0)) =f(x_0) +g(x_0)\expert(x_0) + g(x_0)(\hat\pi(x_0)-\expert(x_0)) = f(x_0)  + g(x_0)\hat\pi(x_0) = \trajj{\hat\pi_0}{x_0}{1} = \hat x_1$.
			Now consider $t>0$.
			\begin{align*}
				\trajs{t+1}{\{\dis_i\}}{x_0}&=\fclexp(\trajs{t}{\{\dis_i\}}{x_0}) +\dis_t, \nn \\ & =\fclexp(\hat x_t ) + g(\hat x_t)(\hpi_t(\hat x_t)- \expert(\hat x_t)),\\
				& = f(\hat x_t ) +g(\hat x_t)\expert(\hat x_t) + g(\hat x_t)(\hpi_t(\hat x_t)- \expert(\hat x_t)), \nn \\
				& = f(\hat x_t) + g(\hat x_t)\hpi_t(\hat x_t)  = \hat x_{t+1}\,.
			\end{align*}
			This concludes the proof of the claim.
		\end{proof}
		Now based on \autoref{claim:traj} together with \autoref{prop:ddi} and Jensen's inequality, we have
		\begin{align}
			\sum_{t=0}^{T} \E\left[\norm{\xstar_t- \hat x_t }\mid \pihat \right] &\leq \gamma^{{1}/{a}} T^{1-{1}/{a}} 
			\left(\sum_{t=0}^{T-1} \E[\norm{
				\dis_t}\mid \pihat]\right)^{{b}/{a }}\nn \\
			&\leq \gamma^{{1}/{a}} T^{1-{1}/{a}} 
			\left(\Lg \cdot \sum_{t=0}^{T-1} \E\left[\norm{\hpi_t(\hat x_t)- \expert(\hat x_t)}\mid \pihat\right]\right)^{{b}/{a}}\,. \label{eq:caused}
		\end{align}
Now, by \autoref{thm:emp_bern}, there is an event $\cE$ of probability at least $1-\delta$ such that under $\cE$: 
		\begin{align}
	\mathbb{E}\left[\|\expert( \hat x_t) -\hat\pi_t( \hat x_t)\|\mid \pihat \right] \leq   \frac{7 \bdd{u}\ln (2T|\Pi|/\delta)}{c n t \ln^2 (t+1)+n}\,, \quad \forall t\geq 0, \label{eq:trainingsuccess}
		\end{align}
 Thus, for $\eps_{n}\coloneqq {14 \bdd{u}}\ln (2T|\Pi|/\delta)/{n}$, we have 
 \begin{align}
  \sum_{t=0}^{T-1} \E\left[\norm{\hpi_t(\hat x_t)- \expert(\hat x_t)}\mid \pihat \right] & =    \mathbb{E} \left[ 		  \norm{\expert( \hat x_0) - \hat \pi_t ( \hat x_0)} \mid \pihat \right]+ \sum_{t=1}^{T-1} \mathbb{E} \left[ 	  \norm{\expert( \hat x_t) - \hat \pi_t ( \hat x_t)} \mid \pihat  \right],\nn \\
  & \leq \frac{\eps_n}{2}   +  \sum_{t=1}^{T-1} \frac{\varepsilon_{n}}{2c t \ln^2 (t+1)} \leq \eps_n.
 \end{align}
	Plugging this into \eqref{eq:caused} implies the desired result.

\section{Proofs from Main Text}
\label{app:proofs}

\subsection{Proof of \autoref{thm:emp_bern}}
\label{sec:pf_lem_emp_bern}

		Let $t\geq 0$, $n_t\geq 2$, and $\ell_t(x)\coloneqq \|\expert(x) -\hat \pi_t(x)\|$.
		By the Empirical Bernstein Inequality \citep[Corollary 5]{maurer2009empirical}, there is an event $\cE$ of probability at least $1-\delta$ such that under $\cE$: 
		\begin{align}
			\E[ \ell_t(\hat x_t)] -\frac{1}{n_t}\sum_{i=1}^{n_t}  \ell_t(\hat x^{(i)}_t) \leq   \sqrt{\frac{2 \widehat V_{t}\ln (|\Pi|/\delta) }{n_t}} + \frac{14  \bdd{u}  \ln (|\Pi|/\delta)}{3(n_t-1)}, \label{eq:firstgar}
		\end{align}
		where $\widehat V_{t}\coloneqq\sum_{i=1}^{n_t} (\ell_t(\hat{x}^{(i)}_t)- \sum_{j=1}^{n_t}\ell_t(\hat{x}^{(j)}_t)/n_t)^2/(n_t-1)$ is the empirical loss variance. Since $\expert$ is in $\Pi$ (\autoref{ass:policy}), it holds that  $\frac{1}{n_t}\sum_{i=1}^{n_t} \ell_t(\hat{x}^{(i)}_t)=0$, which implies have $\widehat V_{t}=0$. This implies that under the event $\cE$, 
		\begin{align}
		    	\mathbb{E} \left[\|\expert( \hat x_t) -\hat\pi_t( \hat x_t)\| \mid \pihat\right] \leq   \frac{7 \bdd{u} \ln (|\Pi|/\delta)}{n_t}. \label{eq:north}
		\end{align}
	Now, a union bound over $t\in [0..T-1]$ with uniform prior implies the desired result.	

\subsection{Proof of \autoref{cor:markov}}
\label{sec:pf_markov}
By \autoref{thm:emp_bern}, there is an event $\cE$ of probability at least $1-\delta$ such that under $\cE$: 
		\begin{align}
	\mathbb{E}\left[\|\expert( \hat x_t) -\hat\pi_t( \hat x_t)\|\mid \pihat \right] \leq   \frac{7 \bdd{u} \ln (2T|\Pi|/\delta)}{c n t \ln^2 (t+1)+n}\,, \quad \forall t\geq 0, \label{eq:trainingsuccess}
		\end{align}
 Let $\eps_{n}\coloneqq {14 \bdd{u}}\ln (2T|\Pi|/\delta)/{n}$. Under the event $\cE$, we have by Markov's inequality: for all $n,$ \begin{align*}
	&\mathbb{P} \left[ \exists t \in [0,T-1],\ 	  \norm{\expert( \hat x_t) - \hat \pi_t ( \hat x_t)} \geq \frac{\varepsilon_{n}}{\delta}  \mid \pihat\right]\nn\\
\leq& \sum_{t=0}^{T-1}	\mathbb{P} \left[  		  \norm{\expert( \hat x_t) - \hat \pi_t ( \hat x_t)} \geq \frac{\varepsilon_{n}}{\delta} \mid \pihat\right], \nonumber \\ 
	 \leq & \frac{\delta}{\varepsilon_{n}}  \mathbb{E} \left[ 		  \norm{\expert( \hat x_0) - \hat \pi_t ( \hat x_0)} \mid \pihat \right]+\frac{\delta}{\varepsilon_{n}} \sum_{t=1}^{T-1} \mathbb{E} \left[ 	  \norm{\expert( \hat x_t) - \hat \pi_t ( \hat x_t)}  \mid \pihat \right],\nonumber \\
	 \leq & \frac{\delta}{2}   + \frac{\delta}{\varepsilon_{n}} \sum_{t=1}^{T-1} \frac{\varepsilon_{n}}{2c t \ln^2 (t+1)}  \leq   \delta, 
\end{align*}
where the probabilities and expectations in the above inequalities are with respect to the randomness in the initial state. Thus, under the event $\cE$, we have with probability at least $1-\delta$ (over the randomness in the initial state),  
\begin{align}
\label{thm:control_affine_old}
			  \norm{\expert( \hat x_t) - \hat \pi_t ( \hat x_t)} \leq \frac{ 14 \bdd{u}\ln (2T|\Pi|/\delta)}{n\delta},\quad \forall t\geq 0.
\end{align}

\subsection{Proof of \autoref{thm:stability}}
\label{pf:stability}

Throughout the proof, we condition on the event $\cE$ described in \autoref{cor:markov} and recall that $\Pr[\cE]\geq 1-\delta$ (randomness w.r.t.~the training process). 
From \eqref{thm:control_affine} in \autoref{cor:markov}, for any $\ell\geq 1$, there exists an event $\ccE\subseteq \mathcal{X}$ of probability at least $1 -\delta/\ell$ (over the randomness in the initial state) under which for all
$t\in[0, \tstar-1]$, 
\begin{align} 
\norm{\rmpc( \hat x_t) - \hat \pi_t ( \hat x_t)} \leq  \frac{ 14 \ell\bdd{u}  \ln (2T\ell|\Pi|/\delta)}{\delta n}, \end{align} 
  Hence, for $\ell\geq 1$, if we choose \begin{align*}
    n\geq \frac{14 \ell \norm{B}\bdd{u}\ln (2T\ell |\Pi|/\delta)}{\eps\delta}\,,
\end{align*} 
then  for all $t=0,\dots \tstar-1$, 
we have $$\norm{\rmpc( \hat x_t) - \hat \pi_t ( \hat x_t)} \leq  \frac{\eps}{\norm{B}}.$$ This implies that $\norm{B(\rmpc( \hat x_t) - \hat \pi_t ( \hat x_t))} \leq  \eps$, and so it follows that under $\ccE$,
\begin{align}
    \forall t\in[0, \tstar-1], \quad \hat x_{t+1} = A \hat x_t + B  \hpi(\hat x_t)  \in A \hat x_t + B \rmpc(\hat x_t ) \oplus \B(\eps) \,, \label{eq:thisevent}
\end{align}
which implies that $\hat x_t$ is an instance of the closed-loop trajectory produced by the robust MPC controller subject to a sequence of  $\eps$-bounded disturbances. Thus, under the event $\ccE$, the trajectory $\{\hat x_{t}\}_{t<\tstar}$ does not violate the constraints and also $\hat x_{t}\in \Olqrref$. 

 We now show that $\hat \tau_{\infty} \leq \tau_{\infty}^\star$ with high probability.
 Recall from \ref{forward-switch} that at the end of each stage $t$, we sample $\ell$ trajectories according to our learned controller $\hpi_{0:t-1}$. 
For $i=1,2,\dots, \ell$, let $\hat x_t^{(i)}$ be the   $i$-th sampled trajectory.
Under $\ccE$, we know that $\hat x_{\tstar}\in \Olqrref$. 
Since $\Pr[\ccE]\geq 1-\delta/\ell$, a union bound over all $i=1,2,\dots,\ell$ yields that there is an event $\cE'$ of probability at least $1-\delta$ s.t.  $\hat x_{\tstar}^{(i)} \in \Olqrref$ for all $i=1,\dots,\ell$.
Hence, under $\cE'$, we have  $\tinf\leq \tstar$. 

We now show the validity of the estimate $\tinf$ of $\tstar$ in the sense that $\hat x_{\tinf} \in \Olqrref$ with high probability.
Let $\pstar_t=\mathbb{P}[\hat x_t \not\in \Olqrref \mid \hmpc]$.
By Bernstein's inequality and a union bound over $t \in[0..T-1]$ with uniform prior, it follows that there exists an event $\cE''$ of probability at least $1-\delta$  such that
\begin{align} \label{exp:prob_ineq}
    \left|\frac{1}{\ntraj}\sum_{i=1}^\ntraj \mathbf{1}\{ \hat x_t^{(i)} \not\in \Olqrref  \} - \pstar_t \right| \leq \sqrt{\frac{2 \pstar_t \cdot \log (T/\delta)}{\ntraj }} + \frac{2 \log (T/\delta)}{3\ntraj}\, ,\quad \forall t\geq 0, 
\end{align}
By definition of $\tinf$, we have $\mathbf{1}\{ \hat x_\tinf^{(i)} \not\in \Olqrref\}=0$, for all $i=1,\dots, \ntraj$, and so under $\cE'\cap \cE''$, we have 
\begin{align*}  
     \pstar_\tinf  &\leq \sqrt{\frac{2 \pstar_\tinf \cdot \log (T/\delta)}{\ntraj }} + \frac{2 \log (T/\delta)}{3\ntraj},\nn \\
     \shortintertext{which after rearranging implies that}
    \mathbb{P}[\hat x_{\tinf} \not\in \Olqrref \mid \tinf,\hmpc] &=\pstar_\tinf  \leq  \frac{10\log (T/\delta)}{\ntraj}
\end{align*}
 Therefore, by choosing $\ntraj$ such that $\ntraj \geq \frac{10\log (T/\delta)}{\delta}$ and $n \geq \frac{14  \ell\norm{B}\bdd{u}\ln (2T\ell |\Pi|/\delta)}{\eps\delta}$, we have that under $\cE'\cap \cE''$, $\hat \tau_{\infty}$ in \ref{forward-switch} satisfies
\begin{align}
\hat \tau_{\infty} \leq \tstar \quad \text{and}\quad \Pr[\hat x_\tinf \in \Olqrref\mid \tinf, \hmpc]\geq 1 -\delta.\nn
\end{align}
Thus, the event $\bar{\cE}\coloneqq\cE \cap \cE'  \cap \cE''$ (which is of probability at least $1-3\delta)$ satisfies the requirements in the theorem's statement. Lastly, as long as $\hat x_\tinf \in \Olqrref$, the LQR controller $\dlqr$ exponentially stabilizes the system while ensuring that the states $\{\hat x_t\}_{t\geq \tinf}$ remain within the positively invariant set $\Olqrref$.  
 And, we have already argued earlier that under the event $\cE'$, the states $\{\hat x_t\}_{t< \tinf}$ do not violate the constraints, which completes the proof.
  
\subsection{Proof of \autoref{thm:performance}}
\label{pf:performance}

We first note a few facts: 
\bli
\item First, the Lipschitz constant of $\piref$ is finite since $\piref$ is continuous piecewise affine function with finitely many pieces (which is a well known property in the MPC literature~\citep{bemporad2002explicit,mayne2001control}). Denote this constant by $\Lpi$.
\item From this fact, it also follows that for each $t$, $\Vref_t$ is continuously differentiable piecewise quadratic with finitely many pieces~\citep{bemporad2002explicit,mayne2001control}. In particular, this fact implies that there exists  $\LVreft>0$ such that $\norm{\nabla \Vref_t(x)} \leq \LVreft \norm{x}$.
Taking $\LVref \coloneqq \max_{t=1,\dots, T} \LVreft$, we thus have $\norm{\nabla_x \Vref_t(x)} \leq \LVref \norm{x}$ for all $t=1,2,\dots,T$.
\item Lastly, from the robust stabilizability result in \autoref{thm:mayne}, as long as the support of $\D$ is bounded almost surely, the closed-loop trajectory produced by the robust MPC controller is bounded subject to any sequence of  $\eps$-bounded disturbances.
Let $\bdd{x}$ be the bound on the closed-loop trajectories of robust MPC under such disturbances and with initial distribution $\D$.

\eli
Using these facts, we prove the following result:
\begin{lemma}
		\label{lem:lipschitzness}
		Let  $\bdd{u} \coloneqq \sup_{u\in \U}\norm{u}$. 
		For each $t=1,2,\dots, T$, the function $\Qref_t$ in \eqref{eq:Qstar} is Lipschitz in the second argument with constant $\LQref=  \LR \bdd{u} +  \LVref \bddx \LB$ for all $\norm{x} \leq \bdd{x}$.
\end{lemma}
\begin{proof}
Fix $t=1,\dots, T$.  
Using the fact that $\norm{\nabla_x \Vref_t(x)} \leq \LVref \norm{x}$, it holds that for $u,v\in\U$, 
$$\Vref_{T-t-1}(Ax+Bu ) - \Vref_{T-t-1}(Ax+Bv) \leq   \LVref \bddx \LB\norm{u-v }.$$ 
Thus, for $u,v\in\U$,   $\Qref_t(x,u)-\Qref_t(x,v) \leq  \LR \bdd{u} \norm{u-v} + \LVref \bdd{x} \norm{u-v}$.
\end{proof}

For the remainder of this proof, we condition on the event $\succE$ of \autoref{thm:stability} and let $x_0$ be an initial state satisfying the events in \eqref{eq:gar}. Now, let  $\cc{t} \coloneqq \traj{\tmpc_{0:t-1}}{x_0}$ and $\xx{t} \coloneqq \traj{\hmpc_{0:t-1}}{x_0}$ where $\tmpc,\hmpc$ are defined in \eqref{switch} of \ref{forward-switch}. By \autoref{thm:stability}, $\cc{t}$ is an instance of the closed-loop trajectory produced by the robust MPC controller subject to a sequence of  $\eps$-bounded disturbances; see \eqref{eq:thisevent} in the proof of \autoref{thm:stability} for details.
Hence, $\norm{\cc{t}}\leq \bdd{x}$ almost surely. 
 By the performance difference lemma (see, e.g.~\citep[Lemma D.12]{foster2020logarithmic}), it then follows that
 \begin{align}
		\Vswitch_T(x_0) - \Vref_T(x_0)  &  =  \sum_{t=0}^{T-1} \Qref_t(\cc{t}, \tmpc_t(\cc{t})) - \Qref_t(\cc{t} , \piref (\cc{t})), \nn  \\
		&\overset{(a)}{=} \sum_{t=0}^{\tinf-1} \Qref_t(\cc{t}, \tmpc_t(\cc{t})) - \Qref_t(\cc{t} , \piref (\cc{t})),\nn\\
			&\leq  \LQref \sum_{t=0}^{\tinf-1} \|\tmpc_t(\cc{t}) - \piref(\cc{t})\| ,\nn\\
			&\leq \LQref  \sum_{t=0}^{\tinf-1}  \left(\norm{\tmpc_t(\cc{t}) - \rmpc(\cc{t})} + \norm{ \rmpc(\cc{t})- \piref(\cc{t})}\right), \nn  \\
		&\leq \LQref  \sum_{t=0}^{\tinf-1}  \left(\eps/\|B\|  +  \norm{ \rmpc(\cc{t})- \piref(\cc{t})} \right), \quad \text{by \eqref{eq:specific}}  \, \nn  \\ 
		&\overset{(b)}{\leq} \LQref  \sum_{t=0}^{\tinf-1}  \left(\eps/\|B\|  +  \kappa \cdot (\Lpi + \|\klqr\|)\eps\right)  \,. \label{eq:last}
		\end{align} 
Here, $(a)$ follows from the fact that $\cc{t} \in \Olqrref$ for $t \geq \tinf$ (by \autoref{thm:stability}), and hence, $\tmpc_t(\cc{t}) = \piref(\cc{t}) = \dlqr(\cc{t})$ due to \eqref{lqr_tinf}. On the other hand, $(b)$ follows by the fact that 
		\begin{align*}
		    \norm{ \rmpc(\cc{t})- \piref(\cc{t})}  &= \norm{\piref(\xref_0(\cc{t}))  + \klqr (\cc{t} -\xref_0(\cc{t})) - \piref(\cc{t})} \leq (\Lpi + \|\klqr\|) \norm{\cc{t}-\xref_0(\cc{t})}\\
		    &\leq (\Lpi + \|\klqr\|)\cdot \kappa \eps\,,
		\end{align*}
		where the last inequality follows from the fact that $\hat x_t -\xref_0(\xx{t}) \in \invamin$ (where $\invamin$ is as in \eqref{eq:min_inv}) together with \autoref{claim:inva} ($\kappa$ is defined in the latter). The desired result now follows by \eqref{eq:last} and the fact that $\tinf \leq \tstar$ (by \autoref{thm:stability}).

 \section{Feasibility and Stability of Robust MPC (Proof of \autoref{thm:mayne})}
\label{pf:mayne}

We begin with some facts that readily follow from the definitions.
\begin{proposition}[{\citep[Proposition 2]{mayne2005robust}}]
   For $x \in (\X_0 \oplus \inva) \cap \X$,  let $\Vrob (x)$ be the optimal cost of \ref{eq:mpc_robust}, and for each $x\in \X_0$, let  $\Vmpcref(x)$ be the optimal cost of \ref{eq:mpcref}.
   Then, the following holds:
\btri
    \item For any $x  \in (\X_0 \oplus \inva) \cap \X$, $\Vrob(x) = \Vmpcref(\xref_0(x))$ and the optimal solution of \ref{eq:mpc_robust} $(\uref_0(x),\uref_1(x),\cdots,\uref_{N-1}(x))$  is equal to the optimal solution of \ref{eq:mpcref} with $\xinit = \xref_0(x)$.   
    \item For all $x\in \inva$,  $\Vrob(x)=0$, $\xref_0(x) =0$, and $(\uref_0(x),\uref_1(x),\cdots,\uref_{N-1}(x)) =(0,0,\dots,0)$ and $\rmpc(x) = \klqr x$.
\eli
\end{proposition}

Next, we present the following result that plays a crucial role in showing the robustness of $\rmpc$.

\begin{proposition}[{\citep[Proposition 1]{mayne2001robustifying}}] \label{prop:invariant}
		Suppose that $x-\barx \in  \inva$, then $x^+ = Ax +B(u+\klqr(x-\barx))+w$ and $\barx^+ = A\barx + Bu$ satisfy $x^+ - \barx^+ \in  \inva$  for all  $w\in\W$.
	\end{proposition}
	\begin{proof}
		From the above definition, we have
		\begin{align*}
			x^+- \barx^+ &=Ax +B(u+\klqr(x-\barx))+w - (A\barx + Bu)  \\
			&= A(x-\barx)  + \klqr(x-\barx) +w \\
			&= A_{\klqr} (x-\barx) + w\,.
		\end{align*} 
Hence by definition of $\inva$ (see \autoref{def:disturbance_inva}), it follows that  $x^+- \barx^+\in \inva$.
	\end{proof}

Now we show that $\rmpc$ stabilizes the system under the presence of disturbances.
\begin{lemma}\label{lem:mayne}
For $x\in (\X_0 \oplus \inva) \cap \X$, let $\xref_0(x)$, $\uref_0(x)$, $\uref_1(x)$, $\cdots$, $\uref_{N-1}(x)$ be the optimal solution to \ref{eq:mpc_robust}, and $\xref_1(x)$, $\xref_2(x)$, $\dots$, $\xref_N(x)$ be the associated state trajectory for the nominal system without disturbances.
Then, for all $x^+ \in (Ax + B\rmpc(x)) \oplus\W$, it holds that $\xref_1(x)$, $\uref_1(x)$, $\uref_2(x)$, $\dots$, $\uref_{N-1}(x)$, $\klqr \xref_N(x)$ is feasible for \ref{eq:mpc_robust} with $\xinit = x^+$ and 
\begin{align*}
    \Vrob(x^+) - \Vrob(x) \leq  -\ell(\xref_0(x), \piref(\xref_0(x)))\,.
\end{align*}
Here $\ell(x,u)$ denotes the stage cost, i.e., $\ell(x,u) = x^\top Qx + u^\top R u$.
\end{lemma}
\begin{proof}[{\bf Proof of \autoref{lem:mayne}}]
    First, note that since $x\in \xref_0(x) \oplus \inva$, it follows from \autoref{prop:invariant} that $x^+ \in \xref_1(x) \oplus \inva$. 
    Since $\xref_0(x)$, $\uref_0(x)$, $\uref_1(x)$, $\cdots$, $\uref_{N-1}(x)$ are feasible solution to \ref{eq:mpc_robust} with $\xinit = x$, it holds that 
    $\xref_1(x)$, $\uref_1(x)$, $\uref_2(x)$, $\dots$, $\uref_{N-1}(x)$ satisfy all the constraints of \ref{eq:mpc_robust} with respect to $\xinit =x^+$.
    Now since, $\xref_N(x) \in \Xter$, we have $\klqr \xref_N(x) \in  \U\ominus \klqr \inva $ and $A_{\klqr} \xref_N(x) \in \Xter$.
    This shows that $\xref_1(x)$, $\uref_1(x)$, $\uref_2(x)$, $\dots$, $\uref_{N-1}(x)$, $\klqr \xref_N(x)$ is feasible for \ref{eq:mpc_robust} with respect to $\xinit = x^+$.
    
    The proof of the value function decrease follows the standard stability analysis of MPC.
    Since $\xref_1(x)$, $\uref_1(x)$, $\uref_2(x)$, $\dots$, $\uref_{N-1}(x)$, $\klqr \xref_N(x)$ is feasible for \ref{eq:mpc_robust} with respect to $\xinit = x^+$, it holds that $\Vrob(x^+)$ is upper bounded by   $\Vmpcref (\xref_1(x))$.
    From the stability result of \ref{eq:mpcref},  it holds that $\Vmpcref (\xref_1(x)) - \Vmpcref (\xref_0(x)) \leq - \ell(\xref_0(x), \piref(\xref_0(x)))$.
    Since $\Vrob(x) = \Vmpcref(\xref_0(x))$, the desired inequality follows.
\end{proof}

With \autoref{lem:mayne}, the rest of the proof follows immediately as follows. 
First, note that there exist absolute constants $C>c>0$ such that for any $x\in (\X_0 \oplus \inva) \cap \X$, the following holds:
\bli
\item From the property of the value function $\Vmpcref$, we have  $\Vrob(x) = \Vmpcref(\xref_0(x)) \geq c \norm{\xref_0(x)}^2$.
\item Moreover,  we have  $\Vrob(x) = \Vmpcref(\xref_0(x)) \leq C \norm{\xref_0(x)}^2$.
\item From \autoref{lem:mayne}, we ahve $\Vrob(x^+) - \Vrob(x) \leq - \ell(\xref_0(x), \rmpc(\xref_0(x))) \leq - c \norm{\xref_0(x)}^2$ for all $x^+ \in (Ax + B\rmpc(x)) \oplus\W$.
\eli
From the above facts, it follows that  $\Vrob(x^+) - \Vrob(x) \leq - c\norm{\xref_0(x)}^2 \leq -\frac{c}{C} \Vrob(x)$, 
which shows $\Vrob(x^+) \leq (1-c/C) \Vrob(x)$. 
Therefore, we obtain
\begin{align*}
    c\norm{\xref_0(x_t)}^2 \leq  \Vrob(x_t) \leq (1-c/C)^t \Vrob(\xinit)  = C (1-c/C)^t \norm{\xref_0(\xinit)}^2\,,
\end{align*}
as desired.
\section{Experimental Details}
\label{app:matrix}

For $d=3$, the $A$ matrix we used is
\begin{align*}
    \begin{bmatrix}
   1.1 &   0.86075747 &  0.4110535\\ 
   0. &    1.1 &    0.4110535 \\
   0. &        0. &      1.1
    \end{bmatrix}\,.
\end{align*} 
 For $d=5$, the $A$ matrix we used is
\begin{align*}
    \begin{bmatrix}
     1.1 & 0.86075747 &  0.4110535 &0.17953273 & -0.3053808 \\
     0. &    1.1 &    0.4110535 &  0.17953273 &-0.3053808 \\ 0. &  0. & 1.1 & 0.17953273 &-0.3053808 \\   0. & 0. &  0. &   1.1 & -0.3053808 \\  0. & 0. & 0. &   0. & 1.1 
    \end{bmatrix}\,.
\end{align*}

\end{document}